\documentclass[11pt,reqno]{amsart}

\usepackage[utf8]{inputenc}
\usepackage[margin=1.25in]{geometry}
\usepackage{hyperref}
\usepackage{appendix}
\usepackage{amsfonts}
\usepackage{amsthm}
\usepackage{amssymb}
\usepackage{stmaryrd} 
\usepackage{amsmath}
\usepackage{amsthm}
\usepackage[dvipsnames]{xcolor}
\usepackage{mathrsfs}
\usepackage{lipsum} 

\theoremstyle{plain}
\newtheorem{theorem}{Theorem}[section]

\newtheorem{lemma}[theorem]{Lemma}

\newtheorem{Definition}[theorem]{Definition}

\theoremstyle{remark}

\newtheorem{remark}[theorem]{Remark}

\usepackage{biblatex} 
\numberwithin{equation}{section}
\title[Non-Hermitian band matrices]{Outliers and bounded rank perturbation for non-Hermitian random band matrices}

\author{Yi HAN}
\address{Department of Mathematics, Massachusetts Institute of Technology, Cambridge, MA.
}
\email{hanyi16@mit.edu}
\thanks{Supported by EPSRC grant EP/W524141/1 and Simons Collaboration grant (601948, DJ)}

\begin{document}

\begin{abstract}

In this work we consider general non-Hermitian square random matrices $X$ 
that include a wide class of random band matrices with independent entries. Whereas the existence of a limiting density is not fully understood for these inhomogeneous models, we show that spectral outliers can be determined under very general conditions when perturbed by a finite-rank deterministic matrix. More precisely, we show that whenever $\mathbb{E}[X]=0,\mathbb{E}[XX^*]=\mathbb{E}[X^*X]=\mathbf{1}$ and $\mathbb{E}[X^2]=\rho\mathbf{1}$, and under mild conditions on sparsity and entry moments of $X$, then with high probability all eigenvalues of $X$ are confined in a neighborhood of the support of the elliptic law with parameter $\rho$. Also, a finite rank perturbation property holds: when $X$ is perturbed by another deterministic matrix $C_N$ with bounded rank and operator norm, then the perturbation induces outlying eigenvalues whose limit depends only on outlying eigenvalues of $C_N$ and $\rho$. This extends the result of Tao \cite{tao2013outliers} on i.i.d. random matrices and O’Rourke and Renfrew on elliptic matrices \cite{article1221} to a family of highly sparse and inhomogeneous random matrices, including all Gaussian band matrices on regular graphs with degree growing faster than $(\log N)^3$. A quantitative convergence rate is also derived. We also consider a class of finite rank deformations of products of at least two independent elliptic random matrices, and show it behaves just as products of i.i.d. matrices.

\end{abstract}

\maketitle

\section{Introduction}

In this work we consider the spectral properties of a very general class of inhomogeneous non-Hermitian random matrices, including a wide class of random band matrices.

To put our results in a proper context, we first recall what is known for a random matrix with i.i.d. entries, which has a fully homogeneous variance profile.

\begin{theorem}\label{zeroththeorem}
Consider an $n\times n$ random matrix $A_n:=\{a_{ij}\}_{1\leq i,j\leq n}$ with i.i.d. elements $a_{ij}$ having mean zero and finite second moment: \begin{equation}\label{firstcondition}\mathbb{E}[a_{ij}]=0 \text{ , and } \mathbb{E}[|a_{ij}|^2]=1.\end{equation} 
\begin{enumerate}
\item (The circular law, see \cite{ WOS:000281425000010} and references therein)  The empirical spectral density of $n^{-1/2}A_n$, defined as $n^{-1}\sum_{j=1}^n\delta_{\lambda_j}$ with $(\lambda_j)_{1\leq j\leq n}$ the eigenvalues of $n^{-1/2}A_n$, converges to the circular law, i.e. the uniform distribution in the unit disk in the complex plane. 
\item (No outliers) Let $\rho(n^{-1/2}A_n)$ denote the spectral radius of $n^{-1/2}A_n$. Assuming the atom distribution $a_{ij}$ has a finite fourth moment: $\mathbb{E}[|a_{ij}|^4]<\infty$, then it was proven in \cite{Bai1986LimitingBO} and \cite{Geman1986THESR} that $\rho(n^{-1/2}A_n)$ converges to 1 almost surely. Later, \cite{ WOS:000435416700013} established the no-outlier result under additional symmetry and moment assumptions, and \cite{bordenave2021convergence} proved convergence in probability under only the finite second moment assumption $\mathbb{E}[|a_{ij}|^2]=1$. A recent work \cite{campbell2024spectral} considered far more general models and showed in particular the no-outlier result for certain elliptic ensembles.

\item (Low rank perturbation, see \cite{tao2013outliers}) Assuming the atom distribution $a_{ij}$ has a finite fourth moment. Fix some $\epsilon>0$ and let $C_n$ be a deterministic matrix with $O(1)$ rank and $O(1)$ operator norm for each $n$, having no eigenvalues with modulus in $[1+\epsilon,1+3\epsilon]$ and $j$ eigenvalues $\lambda_1(C_n),\cdots,\lambda_j(C_n)$ for a $j=O(1)$ in $\{z\in\mathbb{C}:|z|>1+3\epsilon\}$. Then almost surely when $n$ is sufficiently large, there are exactly $j$ eigenvalues $\lambda_1(n^{-1/2}A_n+C_n),\cdots,\lambda_j(n^{-1/2}A_n+C_n)$ of $n^{-1/2}A_n+C_n$ in $\{z\in\mathbb{C}:|z|>1+2\epsilon\}$ and after labeling, $\lambda_i(n^{-1/2}A_n+C_n)=\lambda_i(C_n)+o(1)$ for each $1\leq i\leq j$.
\end{enumerate}

\end{theorem}

\begin{remark}

 The corresponding problem in (3) low rank perturbation  for Hermitian random matrices has been studied in more detail in \cite{peche2006largest}, \cite{capitaine2012central}, \cite{pizzo2013finite}, \cite{feral2007largest}, \cite{benaych2011eigenvalues}, and see also references therein. There is a phase transition called the BBP transition: Consider a Wigner matrix $W_n$, a unit vector $v$ and $\theta>0$, then the largest eigenvalue of $n^{-1/2}W_n+\theta vv^t$ converges to $\theta+\frac{1}{\theta}$ when $\theta>1$ and converges to 2 when $\theta<1$. Item (3) of Theorem \ref{zeroththeorem} shows that when we consider instead a square matrix with i.i.d. entries, then eigenvalues outside the unit circle converge to eigenvalues of the perturbation.

\end{remark}

Inhomogeneous random matrix models have attracted much recent interest, from both the theoretical and applied perspective. Random band matrices are one of the most important inhomogeneous matrix models in modern mathematical physics \cite{bourgade2018random}. In this paper we investigate a non-Hermitian version of random band matrices, which can be defined in the following general form.

\begin{Definition}(Random matrices on regular digraphs with independent entries)\label{definebandmatrix}
Let $N$ be an integer and $d_N\in [1,N]$ be an integer depending on $N$. 

Fix any $d_N$-regular directed graph $G_N=([N],E_N)$ having $N$ vertices and a set $E_N$ of directed edges on $[N]$. 
We further require that there is no self-loop $(x,x)\in E_N$ for any $x\in [N]$.

Consider a family of mean $0$, variance $1$ (real or complex) random variables $g_{xy}$ for each $(x,y)\in E_N$. These random variables are assumed to be independent but not necessarily of the same distribution. Then we define an $N\times N$ random matrix $X_N$ by setting $(X_N)_{xy}=1_{(x,y)\in E_N}g_{xy}$. 

For the almost sure statements in the finite $p$-moment case, we take the standard coupling under which the entries that appear at level $N$ are contained in a single infinite i.i.d. sequence (i.e., we realize the entries appearing in all matrices from a single infinite i.i.d. family, and entries used for the $N$-th matrix are chosen from the first $O(Nd_N)$ variables.)

\end{Definition}Although self-loops are excluded here, we will eventually also consider graphs with self-loops in Theorem \ref{theoremlines404}.

 This definition is sufficiently general as it encompasses a wide class of distinct random matrix models. They include periodically banded matrices where we set $(x,y)\in E_N$ if $\min(|x-y|,N-|x-y|)\leq d_N$ for a bandwidth $d_N$ (up to replacing $d_N$ by a comparable degree parameter). They also include block matrices which are the direct sum of $N/d_N$ blocks of i.i.d. matrix of size $d_N$. They further include the direct sum of block and band matrices, and many other possibilities.

Now we state the main result of this paper, applied to the matrix $(d_N)^{-1/2}X_N$. From the definition, it is easy to see that this matrix is doubly stochastic in the following sense:
$$
\mathbb{E}[(d_N)^{-1/2}X_N]=0,\quad \mathbb{E}[(d_N)^{-1}X_NX_N^*]=\mathbb{E}[(d_N)^{-1}X_N^*X_N]=I_N.
$$

For two sequences of positive real numbers $a_n,b_n$, we use the shorthand notation $a_n\gg b_n$ to mean $\frac{a_n}{b_n}\to\infty$ as $n\to\infty$ and $a_n\ll b_n$ to mean $\frac{a_n}{b_n}\to 0$ as $n\to\infty$.

\begin{theorem}\label{theorem1.1}  

Consider the random matrix model in Definition \ref{definebandmatrix}.
     Assume moreover that
    one of the following three conditions holds:
    \begin{enumerate}
        \item $g_{xy}$ are standard (real or complex) Gaussian variables and $d_N\gg(\log N)^3$, or
        \item $g_{xy}$ are real or complex random variables with mean 0, variance 1 that satisfy
        $$\lim_{N\to\infty} (\log N)^2 \max_{(x,y)\in E_N}{\|(d_N)^{-1/2}g_{xy}\|_\infty}=0,$$where $\|\cdot\|_\infty$ is the essential supremum of a random variable,
        \item $g_{xy}$ are i.i.d. (real or complex) symmetrically distributed, variance 1 random variables such that, for some $p>4$, we have  $\mathbb{E}[|g_{xy}|^p]<\infty$ and $d_N\gg N^\frac{2}{p-2}(\log N)^{\frac{4p}{p-2}}$.
    \end{enumerate} 
    
    Then the random matrix $(d_N)^{-1/2}X_N$ in Definition \ref{definebandmatrix} satisfies the following properties:
    \begin{enumerate}
        \item For any $\epsilon>0$, almost surely for $N$ sufficiently large, $\rho((d_N)^{-1/2}X_N)\leq 1+\epsilon$.
    
    \item Moreover, fix some $\epsilon>0$, let $C_N$ be a deterministic matrix with $O(1)$ rank and operator norm for each $N$, having no eigenvalues with modulus in $[1+\epsilon,1+3\epsilon]$ and $j$ eigenvalues $\lambda_1(C_N),\cdots,\lambda_j(C_N)$ for a $j=O(1)$ in $\{z\in\mathbb{C}:|z|>1+3\epsilon\}$. No restriction is imposed on the eigenvalues of $C_N$ in $\{z\in\mathbb{C}:|z|\leq 1+\epsilon\}$.

    Then almost surely for $N$ sufficiently large, there are exactly $j$ eigenvalues denoted by $\lambda_1((d_N)^{-1/2}X_N+C_N),\cdots,\lambda_j((d_N)^{-1/2}X_N+C_N)$ of $(d_N)^{-1/2}X_N+C_N$ in $\{z\in\mathbb{C}:|z|>1+2\epsilon\}$, and after labeling, they have the following asymptotic: $$\lambda_i((d_N)^{-1/2}X_N+C_N)=\lambda_i(C_N)+O((\log\log N)^{-\frac{3}{4}})$$ for each $1\leq i\leq j$.
    \end{enumerate}

\end{theorem}

\begin{remark}\label{remark1.56}
     A quantitative, high probability statement can be found later in this paper, and a discussion on the quantitative estimates in this theorem are presented in Remark \ref{quantifiedremarks}. 
     Although excluded in Definition \ref{definebandmatrix}, the case where the graph $G_N$ in Definition \ref{definebandmatrix} may have self-loops can be handled via a perturbation argument, see Theorem \ref{theoremlines404} for the Gaussian and bounded case in which the diagonal entries are assumed to be i.i.d.
\end{remark}

\begin{remark}
    In case (3) of this theorem we assumed random variables have a symmetric distribution. This is not fundamentally necessary and is only used in a truncation argument, so after truncation the truncated random variable still has mean 0. This mean zero property helps us to solve a matrix Dyson equation \eqref{matrixdysomequationsfag} without much effort, but we can certainly consider a general centered non-symmetric distribution and show the error of the mean in the truncation just vanishes (under a stronger moment assumption). We omit the details. Note, in contrast, that in the bounded case, case (2), no symmetry condition is imposed.  
\end{remark}

\subsection{The elliptic case}

In this section we state an elliptic analogue of Theorem \ref{theorem1.1}.

Recall that an $N\times N$ complex random matrix $\mathcal{F}_N=(e_{ij})_{i,j=1,\cdots,N}$ is said to be elliptic if each entry $e_{ij}$ has mean zero and variance one, that $\mathbb{E}[e_{ij}e_{ji}]=\rho\in(-1,1)$ for each $i\neq j$, and that all entries of $\mathcal{F}_N$ are independent modulo symmetry constraint. This definition interpolates the i.i.d. case $\rho=0$ and the symmetric case $\rho=1$.
Under some additional assumptions, it is proved in \cite{nguyen2015elliptic} that the empirical eigenvalue density of $N^{-1/2}\mathcal{F}_N$ converges to the elliptic law, that is, the uniform distribution on the ellipse $\mathcal{E}_\rho$ 
\begin{equation}\label{ellipselaw}
\mathcal{E}_\rho:=\{z\in\mathbb{C}:\frac{\operatorname{Re}(z)^2}{(1+\rho)^2}+\frac{\operatorname{Im}(z)^2}{(1-\rho)^2}\leq 1\}
\end{equation}
in the complex plane. 
Assuming further that $e_{ij}$ has a finite fourth moment, \cite{article1221} proved that a finite rank perturbation result also holds for random elliptic matrices, which in the i.i.d. case $\rho=0$ reduces to Theorem \ref{zeroththeorem} (3) (i.e., the result of \cite{tao2013outliers}) and in the Wigner case $\rho=1$ reduces to the result of \cite{peche2006largest}
when the perturbation matrix $C_n$ is self-adjoint.

Analogously to Theorem \ref{theorem1.1}, we now consider random band matrices with elliptic correlations, which generalizes Hermitian random band matrices and random band matrices with independent entries in Definition \ref{definebandmatrix}. As will be explained in Section \ref{proofoutline}, currently there is no written proof on the convergence of the eigenvalue density towards the elliptic law for a random band matrix with elliptic correlation and with bandwidth sublinear in $N$. However, we show in the following theorem that outliers in the banded model can still be determined under elliptic correlation, thus generalizing \cite{article1221} to a wide class of band matrices with elliptic correlations.

In this work we consider any covariance $\rho\in\mathbb{C}:|\rho|\leq 1$ without assuming $\rho$ is a real number. In this case we need to define the region $\mathcal{E}_\rho$ slightly differently as follows, for any $\rho\in\mathbb{C}$, $|\rho|\leq 1,$
\begin{equation}\label{newellipselaw}
\mathcal{E}_\rho:=\{z\in\mathbb{C}:z=x+\rho\bar{x},\quad \text{ for some }|x|\leq 1,\quad x\in\mathbb{C}\}.
\end{equation}
We note that the new definition of $\mathcal{E}_\rho$ degenerates to the old one when $\rho$ is real:
\begin{lemma}\label{triviallemmas}
    Assume that $\rho\in (-1,1)\setminus\{0\}\subset\mathbb{R}$, then the set $\mathcal{E}_\rho$ defined in \eqref{newellipselaw} coincides with the ellipse defined in \eqref{ellipselaw}. When $\rho=0$ the statement is trivial. For $\rho\neq0$, we have 
    $$
\mathcal{E}_\rho=e^{i\arg(\rho)/2}\mathcal{E}_{|\rho|},
    $$ that is, $\mathcal{E}_\rho$ is $\mathcal{E}_{|\rho|}$ rotated by the angle $\arg(\rho)/2$ in the counterclockwise direction across the origin. For $|\rho|=1$, $\mathcal{E}_{\rho}$ is understood through \eqref{newellipselaw} instead of the ellipse formula \eqref{ellipselaw}.
\end{lemma} 
This lemma follows from routine computations, and see Section \ref{proofofmainresults} for the proof. 

After submitting the first draft of the manuscript, we learned that elliptic random matrices with general $\rho\in\mathbb{C}$ are previously considered in \cite{girko1986elliptic} where elliptic law was introduced, and in \cite{gotze2015minimal} where they proved the convergence of empirical spectral measure to $\mathcal{E}_\rho$ for any $|\rho|<1$ and for the (dense) elliptic matrix. (Instead of the banded matrix we consider here.)

We continue with the study of outliers. For any $\epsilon>0$ we denote by $\mathcal{E}_{\rho,\epsilon}$ the following region which is a neighborhood of $\mathcal{E}_\rho$ in the complex plane: 
$$
\mathcal{E}_{\rho,\epsilon}:=\{z\in\mathbb{C}:\operatorname{dist}(z,\mathcal{E}_\rho)\leq\epsilon\}.
$$ We also use the notation $\mathbb{D}$ to mean the unit disk in the complex plane $\mathbb{D}:=\{z\in\mathbb{C}:|z|\leq 1\}$, and for any $\epsilon>0$, $(1+\epsilon)\mathbb{D}$ denotes the set $\{z\in\mathbb{C}:|z|\leq 1+\epsilon\}$. Clearly, $\mathcal{E}_{0,\epsilon}=(1+\epsilon)\mathbb{D}$.

We first define the model of elliptic random matrices we consider.

\begin{Definition}(Elliptic random matrices on regular graphs)\label{ellipticdef}
    Let $N\in\mathbb{N}_+$ and $d_N\in[1,N]$. Consider a family of undirected $d_N$-regular graphs $G_N=([N],E_N)$ such that no self-loops $(x,x),x\in[N]$ exist in $E_N$.

Fix a covariance parameter $\rho\in\mathbb{C}:|\rho|\leq 1$.
Define a random matrix $X_N$ on $[N]^2$ such that for any $x<y$, $(x,y)\in E_N$, we set $((X_N)_{(x,y)},(X_N)_{(y,x)})$ to be a pair of mean $0$, variance $1$ random variables such that $$\mathbb{E}[(X_N)_{(x,y)}(X_N)_{(y,x)}]=\rho.$$
If $(x,y)\notin E_N$ we set the pair to be $(0,0)$. We assume that all the random variables $((X_N)_{(x,y)},(X_N)_{(y,x)})_{(x,y)\in E_N}$ are mutually independent over the edges $(x,y)\in E_N$. 

Under condition (3) below, when deriving almost sure events in the $N\to\infty$ limit, it is understood that the entry pairs defined on all edge pairs appearing in all matrices form a single infinite i.i.d. family of pairs $(g_{1,k},g_{2,k})$, and after an arbitrary labeling of edges, the pairs used for the $N$-th matrix are taken from the first $O(Nd_N)$ pairs.

\end{Definition}

We then state the main result concerning elliptic band matrices.
\begin{theorem}\label{TheEllipticTheorem} Let $X_N$ be an elliptic random matrix as in Definition \ref{ellipticdef}.
For the entry distributions, we further assume that  one of the following three conditions holds: \begin{enumerate}
    \item   $d_N\gg(\log N)^3$ and the matrix $X_N$ has jointly Gaussian entries;  \item  $d_N\gg(\log N)^4$ and the complex-valued entries of $X_N$ satisfy $$\lim_{N\to\infty} (\log N)^2 (d_N)^{-1/2}\max_{(x,y)\in E_N}{\|\max(|(X_N)_{(x,y)}|,|(X_N)_{(y,x)}|)\|_\infty}=0;$$ \item For each $(x,y)\in E_N:x<y$, the pair $((X_N)_{(x,y)},(X_N)_{(y,x)})$ has the same law as $(g_1,g_2)$, where $g_1,g_2$ are two complex-valued symmetric random variables of mean $0$, variance $1$, having the same distribution: $g_1\stackrel{\text{law}}{=}g_2$ and $\mathbb{E}[g_1g_2]=\rho$. We require that, for some $p>4$ and $d_N\gg N^\frac{2}{p-2}(\log N)^{\frac{4p}{p-2}}$, we have $\mathbb{E}[|g_i|^p]<\infty,i=1,2$.
\end{enumerate}
Then we have the following conclusion: \begin{enumerate}
    \item For any $\epsilon>0$, almost surely as $N$ tends to infinity there is no eigenvalue of $(d_N)^{-1/2}X_N$ in $\mathbb{C}\setminus\mathcal{E}_{\rho,\epsilon}$.

\item Moreover, fix $\epsilon>0$ and consider a family $C_N$ of deterministic $N$ by $N$ matrices with $O(1)$ rank and operator norm, having no eigenvalues that satisfy 
$$
\lambda_i(C_N)+\frac{\rho}{\lambda_i(C_N)}\in\mathcal{E}_{\rho,3\epsilon}\setminus\mathcal{E}_{\rho,\epsilon}\text{ and } |\lambda_i(C_N)|>1,
$$
and having $j=O(1)$ eigenvalues $\lambda_1(C_N),\cdots,\lambda_j(C_N)$ such that $$\lambda_i(C_N)+\frac{\rho}{\lambda_i(C_N)}\in\mathbb{C}\setminus\mathcal{E}_{\rho,3\epsilon},\text{ and }|\lambda_i(C_N)|> 1,\quad\forall 1\leq i\leq j.$$ 

Then almost surely for $N$ large enough, there are precisely $j$ eigenvalues of $(d_N)^{-1/2}X_N+C_N$: $\lambda_i((d_N)^{-1/2}X_N+C_N)$, $i=1,\cdots,j$ in $\mathbb{C}\setminus\mathcal{E}_{\rho,2\epsilon}$ that satisfy, after a suitable relabeling, $$\lambda_i((d_N)^{-1/2}X_N+C_N)=\lambda_i(C_N)+\frac{\rho}{\lambda_i(C_N)}+O((\log\log N)^{-\frac{3}{4}})$$ for each $1\leq i\leq j$.
\end{enumerate}

\end{theorem}

\subsection{A master theorem}\label{section247}

In this paper we will not prove Theorem \ref{theorem1.1} and \ref{TheEllipticTheorem} directly, but rather deduce them as special cases of the following master theorem. 

Throughout this paper we use the notation $\mathbf{1}$ to denote an identity matrix of finite dimension, and its dimension will be clear from the context.

We also use the notation “$\operatorname{id}$” for the identity operator on the space of $N\times N$ square matrices.

A major observation here is that the regular graph structures in Theorem \ref{theorem1.1} and \ref{TheEllipticTheorem} are not the most intrinsic condition: for an $N$ by $N$ random matrix $X$ with $\mathbb{E}[X]=0,$ the intrinsic condition guaranteeing ellipticity should be 
$\mathbb{E}[X^2]=\rho \mathbf{1}$ for some fixed constant $\rho$. The assumption that the variance profile is doubly stochastic, namely $\mathbb{E}[XX^*]=\mathbb{E}[X^*X]=\mathbf{1}$, is sufficiently general to cover both the homogeneous case of a standard i.i.d. matrix and a standard elliptic matrix, and the band matrix case considered in Theorem \ref{theorem1.1} and \ref{TheEllipticTheorem}. We will show this is indeed the case: for a very general random matrix $X$ with mean zero and satisfying some constraints, then its behavior in terms of outliers is completely determined whenever $\mathbb{E}[X^2]=\rho \mathbf{1}$ and $\mathbb{E}[XX^*]=\mathbb{E}[X^*X]=\mathbf{1}$ for some $|\rho|\leq 1$. This generalizes the Hermitian case where $\rho=1$, as considered in Theorem 2.7 of \cite{bandeira2024matrix}.

Before presenting our master theorem, we introduce some notation on very general random matrices from \cite{bandeira2023matrix} and \cite{brailovskaya2022universality}.

In the Gaussian case, consider the following random matrix model as in \cite{bandeira2023matrix}, Section 1. Consider fixed matrices $A_0,\cdots,A_n\in M_N(\mathbb{C})$, where $M_N(\mathbb{C})$ denotes the set of $N$ by $N$ square matrices with complex entries. Let $g_1,\cdots,g_n$ be independent standard real Gaussian variables with mean 0 and variance 1. Consider a canonical model \begin{equation}\label{gaussianmatrixform}
    X:=A_0+\sum_{i=1}^n g_i A_i. 
\end{equation} (There are certainly many different ways to form a decomposition as \eqref{gaussianmatrixform}. And for our main usage in this paper, we will work in the case where the nonzero entries of $A_i$ have disjoint supports, so that directly replacing the $g_i$ by other random variables will automatically give rise to the non-Gaussian matrix. The use of independent Gaussian variables and many $A_i$ here is essential: replacing the sum by a simple Gaussian $gA$ is plausible but will hide the nice independence structures among entries of $X$, making our applications infeasible.)

We need to introduce a few more parameters on $X$.

For a Gaussian matrix $X$ \eqref{gaussianmatrixform} we define its covariance profile $\operatorname{Cov}(X)\in M_{N^2}(\mathbb{C})$ via 
$$
\operatorname{Cov}(X)_{ij,kl}=\sum_{s=1}^n (A_s)_{ij}\overline{(A_s)_{kl}} 
$$ and we now define 
$$\begin{aligned}
&v(X)^2=\|\operatorname{Cov}(X)\|=\sup_{\operatorname{Tr}|M|^2\leq 1}\sum_{s=1}^n |\operatorname{Tr}[A_sM]|^2,\\
&\sigma(X)^2:=\|\mathbb{E}[(X-\mathbb{E}X)^*(X-\mathbb{E}X)]\|\vee\|\mathbb{E}[(X-\mathbb{E}X)(X-\mathbb{E}X)^*]\|,\\
&\sigma_*(X)^2:=\sup_{\|v\|=\|w\|=1}\mathbb{E}[|\langle v,(X-\mathbb{E}X)w\rangle|^2],
\end{aligned}$$
and finally define $$\tilde{v}(X)^2:=v(X)\sigma(X).$$

\begin{remark}[The normalization forces $\sigma(X)=1$]\label{sigmaone}
The matrices in both master theorems below are centered and normalized so that
\(\mathbb{E}[XX^*]=\mathbb{E}[X^*X]=\mathbf{1}\). Consequently,
\[
\sigma(X)^2
=\|\mathbb{E}[X^*X]\|\vee\|\mathbb{E}[XX^*]\|
=1,
\qquad
\sigma(X)=1,
\qquad
\widetilde v(X)^2=v(X).
\]
Thus $\sigma(X)$ carries no additional information in the hypotheses of the master
theorems, and we state their smallness assumptions directly in terms of $v(X)$.
We retain $\sigma$ and $\widetilde v$ in the general concentration estimates below,
where the matrices need not satisfy this normalization.
\end{remark}

We first give a heuristic explanation for why these quantities are useful. For a general matrix, $\sigma(X)$ measures its overall size, while $v(X)$ and $\sigma_*(X)$ measure the size of the individual building blocks $A_i$. Under the normalization of our master theorems, however, Remark \ref{sigmaone} gives $\sigma(X)=1$ and $\widetilde v(X)^2=v(X)$. The smallness condition \eqref{smallness} on $v(X)$ and $\sigma_*(X)$ ensures that $X$ is made up of sufficiently many independent blocks and that each block has small size (this matches our requirement that the graph degree $d_N$ is at least $(\log N)^{O(1)}$ for specific models).

Our general theorem in the Gaussian case is as follows
\begin{theorem}\label{gaussianmastertheorem} Fix $\rho\in\mathbb{C}$ with $|\rho|\leq 1$.
    Consider an $N
    \times N$ Gaussian random matrix $X$ with \begin{equation}\label{regularityconditions}\mathbb{E}[X]=0,\quad \mathbb{E}[XX^*]=\mathbb{E}[X^*X]=\mathbf{1},\quad \mathbb{E}[X^2]=(\rho+D_N)\mathbf{1},\end{equation} where $D_N\in\mathbb{C}$ are complex numbers such that $\lim_{N\to\infty}(\log N)D_N=0,$ and that $|\rho+D_N|\leq 1$ for all $N$.

    Assume further that \begin{equation}\label{smallness}v(X)=o\big((\log N)^{-\frac{3}{2}}\big),\quad \sigma_*(X)\leq (\log N)^{-\frac{3}{2}}.\end{equation}

    Then for any $\epsilon>0$, we can 
 find a constant $C_1>0$ depending only on $\rho$ and $\epsilon$ (and, for the second conclusion, on fixed upper bounds for $\operatorname{rank}(C_N)$ and $\|C_N\|_{op}$), such that
    \begin{enumerate}
        \item 
    With probability at least $1-C_1N^{-1.5}$ there is no eigenvalue of $X$ lying in $\mathbb{C}\setminus\mathcal{E}_{\rho,\epsilon}$.
\item Consider deterministic $N$ by $N$ matrices $C_N$ of rank $O(1)$ and operator norm $O(1)$, having no eigenvalues that satisfy 
$$
\lambda_i(C_N)+\frac{\rho}{\lambda_i(C_N)}\in\mathcal{E}_{\rho,3\epsilon}\setminus\mathcal{E}_{\rho,\epsilon}\text{ and } |\lambda_i(C_N)|>1,
$$
and having $j=O(1)$ eigenvalues $\lambda_1(C_N),\cdots,\lambda_j(C_N)$ such that $$\lambda_i(C_N)+\frac{\rho}{\lambda_i(C_N)}\in\mathbb{C}\setminus\mathcal{E}_{\rho,3\epsilon},\text{ and }|\lambda_i(C_N)|> 1,\quad\forall 1\leq i\leq j.$$ 

Then with probability at least $1-C_1N^{-1.4}$, there exist precisely $j$ eigenvalues $\lambda_i(X+C_N)$, $i=1,\cdots,j$ of $X+C_N$ in $\mathbb{C}\setminus\mathcal{E}_{\rho,2\epsilon}$ and, after a possible relabeling of the eigenvalues, we have $$\lambda_i(X+C_N)=\lambda_i(C_N)+\frac{\rho}{\lambda_i(C_N)}+O((\log\log N)^{-3/4})$$ for each $1\leq i\leq j$.
\end{enumerate}

\end{theorem}

\begin{remark}\label{quantifiedremarks}
We have presented our main theorem with quantified probability and convergence estimates. These estimates are not meant to be optimal and are of illustrative use only. The probability estimate $1-C_1N^{-1.4}$ can easily be upgraded to be $1-C_1N^{-p}$  for any $p\in\mathbb{N}_+$, and we use the former as it is already sufficient for applying Borel-Cantelli lemma. 

The error term $(\log\log N)^{-3/4}$ is more delicate and can be replaced by any quantity that vanishes strictly slower than $(\log N)^{-p}$ and any $p>0$. We choose the term $(\log\log N)^{-3/4}$ for notational simplicity only. This appears to be the first time that an explicit convergence rate is proven for elliptic random matrices.

An interesting problem is whether the error term can decay faster than $(\log N)^{-p}$: we believe this cannot hold unconditionally but instead relies strongly on the structure of the random matrix $X$. For a square matrix with i.i.d. entries, fluctuations of outlying eigenvalues on the scale $N^{-1/2}$ around their deterministic limit are studied in \cite{bordenave2016outlier}. As we consider far more general non-Hermitian random matrices, we do not expect a similar pattern of fluctuations should hold for typical band matrices considered in this paper.

\end{remark}

In the general case, consider the following random matrix model as in \cite{brailovskaya2022universality}

\begin{equation}\label{ageneralX}
    X:=Z_0+\sum_{i=1}^n Z_i
\end{equation} where $Z_0\in M_N(\mathbb C)$ is deterministic and $Z_1,\cdots,Z_n$ are independent $N$ by $N$ random matrices having mean zero. For this general matrix model $X$ we define its covariance matrix $\operatorname{Cov}(X)$ as an $N^2\times N^2$ matrix satisfying 
\begin{equation}
    \operatorname{Cov}(X)_{ij,kl}:=\mathbb{E}[(X-\mathbb{E}X)_{ij}\overline{(X-\mathbb{E}X)_{kl}}].
\end{equation}

Then we can define $v(X),\sigma(X)$ and $\sigma_*(X)$ analogously as in the Gaussian model.
We also define the following function for $X$, where $\|Z_i\|$ is the operator norm of $Z_i$:
\begin{equation} R(X):=\left\|\max_{1\leq i\leq n}\|Z_i\|\right\|_\infty,\quad 
    \bar{R}(X):=\mathbb{E}\left[\max_{1\leq i\leq n}\|Z_i\|^2\right]^\frac{1}{2},
\end{equation}
where we use the notation $\|\cdot\|$ for the operator norm and the notation $\|Y\|_\infty$ for the essential supremum of a random variable $Y$.

Then our master theorem for non-Gaussian $X$ is as follows. In view of Remark \ref{sigmaone}, the bound on $\widetilde v(X)$ can be stated directly as a bound on $v(X)$. In addition to bounds on $v(X)$ and $\sigma_*(X)$, which ensure that $X$ is made up of many blocks with small variance each, we also need a bound on $R(X)$ to ensure that every individual component is small in operator norm; while this is automatic in the Gaussian case, it must be imposed in the general case.

\begin{theorem}\label{nongaussianmastertheorem}
    Fix $\rho\in\mathbb{C}$ with $|\rho|\leq 1$. Let $X$ be an $N$ by $N$ random matrix in the form \eqref{ageneralX}. Assume that  $$\mathbb{E}[X]=0,\quad \mathbb{E}[XX^*]=\mathbb{E}[X^*X]=\mathbf{1},\quad \mathbb{E}[X^2]=(\rho+D_N)\mathbf{1},$$ where $D_N$ are complex numbers satisfying $\lim_{N\to\infty}(\log N)D_N=0$ and that $|\rho+D_N|\leq 1$ for all $N$. Assume that $$v(X)\leq (\log N)^{-2},\quad \sigma_*(X)\leq (\log N)^{-\frac{3}{2}},\quad R(X)\ll(\log N)^{-2}.$$ Then all the conclusions in Theorem \ref{gaussianmastertheorem} continue to hold for $X$.
\end{theorem}

\begin{remark}
    The constant $D_N\to 0$ in Theorem \ref{gaussianmastertheorem} and \ref{nongaussianmastertheorem} is used to take care of the diagonal entry of $X$ in Theorem \ref{theoremlines404}. 
\end{remark}

\begin{remark} Theorem \ref{TheEllipticTheorem} weakens a number of hypotheses that are traditionally imposed in the elliptic random matrix literature.
    In prior works \cite{article1221}, \cite{nguyen2015elliptic} and \cite{alt2022local}, they consider an elliptic matrix where each pair $(g_{ij},g_{ji})$ is independently distributed as a pair $(g_1,g_2)$, where $(g_1,g_2)$ belongs to the so-called $(\mu,\rho)$ family (\cite{nguyen2015elliptic}, Definition 1.6). This definition of $(\mu,\rho)$ family requires the real and imaginary parts of $g_1,g_2$ are independent and the covariance of the real and imaginary parts of $g_i$ should have a very specific form. Theorem \ref{TheEllipticTheorem} states that such restrictions are not necessary, and we do not even need the pairs $(g_{ij},g_{ji})$ and $(g_{kl},g_{lk})$ to be identically distributed: their real and imaginary parts can have a completely different covariance matrix for different edges $(i,j),(k,l)\in E$. As another weakening, we can consider any covariance $\rho\in\mathbb{C},|\rho|\leq 1$ without assuming $\rho$ is real, unlike in \cite{article1221}, \cite{nguyen2015elliptic} and \cite{alt2022local}. In these cited works, the cases $\rho=\pm 1$ are excluded as the least singular value bound in \cite{nguyen2015elliptic} no longer holds: in the Wigner case $\rho=1$ this still holds with a separate proof \cite{nguyen2012least} whereas in the anti-symmetric case $\rho=-1$ the matrix is always singular in odd dimensions. In our paper we do not rely on least singular value estimates as in \cite{nguyen2015elliptic}, so we treat $|\rho|=1$ in a unified way as for $|\rho|<1$, and in particular we cover the anti-symmetric case $\rho=-1$ which was largely unexplored before.

    \end{remark}

\begin{remark}\label{remark1.31.3aga}
    Theorem \ref{theorem1.1} and \ref{TheEllipticTheorem} also cover (homogeneous) sparse i.i.d. matrices and elliptic matrices. Indeed, we may take $G$ to be the complete graph on $[N]$ and take random variables $g_{xy}$ to be $g_{xy}=b_{xy}h_{xy}$ where $b_{xy}$ is a $\operatorname{Ber}(\frac{k_N}{N})$ random variable taking value 1 with probability $\frac{k_N}{N}$ and 0 otherwise, and $h_{xy}$ are bounded random variables having mean 0 and variance 1, independent from $b_{xy}$. Then clearly $(N/k_N)^{1/2}g_{xy}$ has mean $0$ and variance $1$, and as long as $k_N>(\log N)^5$ we have that $(N/k_N)^{1/2}g_{xy}$ satisfy the assumptions of Theorem \ref{theorem1.1}, (2) with $g_{xy}$ replaced by $(N/k_N)^{1/2}g_{xy}$ and $d_N:=N-1$. Indeed, the proof of Theorem \ref{theorem1.1} (2) allows distributions of $g_{xy}$ to be $N$-dependent so long as all quantitative estimates are satisfied. This establishes finite rank perturbations for $k_N$-sparse (homogeneous) i.i.d. matrices whenever $k_N\gg(\log N)^5$, which also appears to be new. 
    
    We note that for homogeneous sparse i.i.d. matrices with $k_N$ much smaller than $\log N$, including fixed $k_N$, an analogous outlier result was proved in \cite{han2024finite} for perturbations with a uniformly bounded number of nonzero entries. The same work also treats certain classes of perturbations for dense i.i.d. matrices under only a finite second moment assumption. These results use the method of characteristic functions, adapted from computations in \cite{bordenave2021convergence} and \cite{coste2023sparse}. The results in \cite{han2024finite} are not accessible by techniques in this paper as this paper relies on a Hermitization procedure, and a Hermitian matrix with only finite second moment entries or constant average degree typically has unbounded operator norm. On the other hand, general band matrices are not accessible by techniques from \cite{han2024finite} either, as in this case the characteristic function method from \cite{bordenave2021convergence} and  \cite{coste2023sparse} breaks down. Also, in our master theorem \ref{gaussianmastertheorem} we can consider a very general square matrix $X$ without assuming independence of entries, and thus covers many interesting new applications (Theorem \ref{outliersofproducts}); but the characteristic function method requires strongly that all entries are independent.

For the elliptic case, the assumptions of Theorem \ref{TheEllipticTheorem} cover (homogeneous) sparse elliptic random matrices, which is the entrywise product of a standard elliptic random matrix as considered in \cite{article1221}, together with a symmetric random matrix $B=(b_{xy})$ with $B_{xy}=B_{yx}\sim \operatorname{Ber}(\frac{k_N}{N})$, for some $
k_N\gg(\log N)^5$, and that all entries of $B$ are independent modulo symmetry. The conclusion of Theorem \ref{TheEllipticTheorem} appears to be new even for such (homogeneous) sparse elliptic matrices as the fourth moment of each entry is not bounded, lying outside the assumptions in \cite{article1221}.

\end{remark}

\begin{remark}A somewhat orthogonal set of assumptions on non-Hermitian random matrices with inhomogeneous profile, is to assume an arbitrary variance profile but with a universal lower and upper bound on the variance of each entry. See \cite{alt2018local} and \cite{alt2021spectral} for local laws and sharp spectral radius bounds in this inhomogeneous case. The assumption that the variances of all entries are bounded from below appears to be crucial and excludes random band matrices (see \cite{cook1article},\cite{cook2article} for a weakening of this hypothesis, but still for dense matrices). 
    Whereas allowing sparsity, our approach instead relies strongly on the regularity hypothesis \eqref{regularityconditions} (the variance matrix is doubly stochastic) to yield a simple solution to the matrix Dyson equation \eqref{matrixdysomequationsfag}. To freely change the variance profile of each entry, one needs to find solutions to the matrix Dyson equation \eqref{matrixdysomequationsfag} with general coefficients, but this seems currently out of reach when the underlying graph $G_N$ is very sparse.
\end{remark}

\subsection{Applications: product of independent elliptic matrices / self loops}   Random matrix products are a central topic in random matrix theory, and the study of outliers can likewise be generalized to the product setting. Suppose we consider a class of independent non-Hermitian random matrices $X_N^{i},i=1,\cdots,m$ satisfying Definition \ref{definebandmatrix}, we can investigate the product matrix $(d_{N,1})^{-1/2}X_N^1\cdots(d_{N,m})^{-1/2}X_N^m$ and its finite-rank deformation $((d_{N,1})^{-1/2}X_N^1+A_N^1)\cdots((d_{N,m})^{-1/2}X_N^m+A_N^m)$, which is a multi-matrix version of Theorem \ref{theorem1.1}.

The work \cite{coston2020outliers} considered the product of independent matrices with i.i.d. entries and showed a similar result as in the one matrix case, that is, the conclusion of Theorem \ref{theorem1.1} continues to hold for certain deformations of product i.i.d. matrices. They also mentioned a similar problem: the product of independent elliptic matrices with general $\rho_i\in(-1,1)$. Indeed, \cite{o2015products} showed a remarkable universality phenomenon: the empirical spectral distribution of a product of ($m\geq 2$) independent elliptic random matrices converges to the $m$-fold circular law, just as in the case of a product of $m$ independent matrices with i.i.d. entries. They also used ideas in free probability to explain why this happens: the product of two free elliptic elements is $R$-diagonal, and thus has an isotropic Brown measure. However the tools in \cite{coston2020outliers} did not seem strong enough to study the outliers of product elliptic matrices due to strong dependence in the entries. After submitting the manuscript, we learned that the empirical spectral density of product of elliptic matrices was also previously studied in \cite{gotze2014one}.

We show a far-reaching strengthening of the result in \cite{coston2020outliers}: exactly the same result continues to hold when we consider the product of independent elliptic random matrices with parameters $\rho_i\in\mathbb{C}:|\rho_i|\leq 1$, and the conclusion continues to hold as we consider non-Hermitian band matrices as in Theorem \ref{theorem1.1} and \ref{TheEllipticTheorem}. Previously such results were only proven for random matrices with i.i.d. entries.

The main result for product elliptic matrices is the following theorem, which can be regarded as a corollary of our master theorem, Theorems \ref{gaussianmastertheorem} and \ref{nongaussianmastertheorem}.

\begin{theorem}\label{outliersofproducts} Let $m\geq 2$ be a fixed integer, and let $(d_{N,1})^{-1/2}X_N^1,\cdots,(d_{N,m})^{-1/2}X_N^m$ be independent $N\times N$ random matrices satisfying one of the following
\begin{itemize}
    \item (Independent case) Each $(d_{N,i})^{-1/2}X_N^i$ satisfies the assumption of Theorem \ref{theorem1.1},
 \item (Elliptic case) Each $(d_{N,i})^{-1/2}X_N^i$ satisfies the assumption of Theorem \ref{TheEllipticTheorem} with some $|\rho_i|\leq 1$, \end{itemize}and the entries of $X_N^i$ should simultaneously satisfy one of the three constraints on entries in Theorem \ref{theorem1.1} or Theorem \ref{TheEllipticTheorem} (i.e., for $i=1,\cdots,m$ they are simultaneously either Gaussian, or bounded, or having bounded $p$-th moment). Then 
 \begin{enumerate}\item For any $\epsilon>0$, almost surely as $N$ tends to infinity there is no eigenvalue of $$D_N^m:=(d_{N,1}\cdots d_{N,m})^{-1/2}X_N^1
\cdots X_N^m$$ lying outside of $(1+\epsilon)\mathbb{D}$. \item Moreover we have the following finite rank perturbation result. Let $A_N^1,\cdots,A_N^m$ be deterministic $N\times N$ matrices with $O(1)$ rank and operator norm. Then consider $$D_N^{m,1}:=\prod_{k=1}^m \left((d_{N,k})^{-1/2}X_N^k+A_N^k\right),\quad A_N:=\prod_{k=1}^m A_N^k.$$ Fix an $\epsilon>0$ and assume that there is no eigenvalue of $A_N$ lying in $\{z\in\mathbb{C}:1+\epsilon\leq|z|\leq 1+3\epsilon\}$ and that there are $j=O(1)$ eigenvalues $\lambda_1(A_N),\cdots,\lambda_j(A_N)$ outside $(1+3\epsilon)\mathbb{D}$. Then almost surely as $N$ tends to infinity, there are exactly $j$ eigenvalues $\lambda_1(D_N^{m,1}),\cdots,\lambda_j(D_N^{m,1})$ lying outside $(1+2\epsilon)\mathbb{D}$ and after rearrangement, we have
    $$\lambda_i(D_N^{m,1})=\lambda_i(A_N)+o(1),\quad i=1,\cdots,j.$$
\end{enumerate}    
\end{theorem}

The special case $m=2$ and $\rho_1=\rho_2=1$ implies the eigenvalues of the product of two independent Wigner matrices are contained in $(1+\epsilon)\mathbb{D}$ with high probability for any $\epsilon>0$.

As a side remark, in Theorem \ref{outliersofproducts}, we can weaken the independence assumption to only requiring the neighboring matrices $X_N^j,X_N^{j+1}$ to be independent for all $j$, with indices $j$ understood modulo $m$. The latter weakened assumption is often still sufficient to justify the assumptions of Theorem \ref{gaussianmastertheorem}.

\begin{remark}
    For any non-commutative polynomial $P$ with $m+n$ non-commutative variables, one may consider the outliers of $P$ evaluated at $m$ independent random matrices and $n$ deterministic matrices. One also assumes the $n$ deterministic matrices have a strong limit in free probability. For random matrices with i.i.d. entries this was investigated in \cite{belinschi2021outlier}. For non-Hermitian band matrices, it would be possible to generalize Theorem \ref{outliersofproducts} to outliers of any polynomial $P$, but this requires considerably more effort and is left for further study.
\end{remark}

\subsubsection{Self-loops}
In Theorem \ref{theorem1.1} and \ref{TheEllipticTheorem} we have assumed that no self-loops exist on the graph $G_N$, but this is not fully necessary. In the next theorem we consider the case with self-loops:

\begin{theorem}\label{theoremlines404}
    Let $N\in\mathbb{N}_+$ and $d_N\in[1,N]$. Consider a family of undirected $d_N$-regular graphs $G_N=([N],E_N)$ where each self-loop $(x,x),x\in [N]$ is in $E_N$. Here $d_N$-regular means that the sum of each row of the adjacency matrix equals $d_N$. Fix a covariance parameter $\rho\in\mathbb{C}:|\rho|\leq 1$. Define a random matrix $X_N$ on $[N]^2$ such that for each $(x,y)
    \in E_N,x<y$, we set $((X_N)_{x,y},(X_N)_{y,x})$ to be a pair of mean 0, variance 1 random variables with $\mathbb{E}[(X_N)_{x,y}(X_N)_{y,x}]=\rho$, and different pairs of entries are mutually independent. We further assume that the diagonal entries $(X_N)_{x,x}$ are i.i.d. copies of a mean 0, variance 1 random variable $\xi$ and are independent from the off-diagonal entries. If $(x,y)\notin E_N$ we set the pair of entries to be $(0,0)$. Assume that one of the following two conditions holds:
\begin{enumerate}
    \item   $d_N\gg(\log N)^3$ and the matrix $X_N$ has jointly Gaussian entries;  \item  $d_N\gg(\log N)^4$ and the complex-valued entries of $X_N$ satisfy $$\lim_{N\to\infty} (\log N)^2 (d_N)^{-1/2}\max_{(x,y)\in E_N}{\|\max(|(X_N)_{(x,y)}|,|(X_N)_{(y,x)}|)\|_\infty}=0.$$
    For self loops $(x,x)\in E_N$ this condition reads $\lim_{N\to\infty}(\log N)^2(d_N)^{-1/2}\|\xi\|_\infty=0.$
    \end{enumerate}
    Then the same conclusion as in Theorem \ref{TheEllipticTheorem} holds verbatim in this case with self-loops.
\end{theorem}

We can also consider a finite $p$-th version of the theorem (where we also assume the same $p$-th moment bound on $\xi$ and assume $\xi$ is symmetrically distributed), and the proof should follow the same line. We omit the details for simplicity.

\subsection{A literature review of other techniques and a brief proof outline}\label{proofoutline}
Now we briefly discuss the context of the problem investigated, and illustrate the major difficulties and key proof strategies.

For random band matrices with independent entries with bandwidth $d_N$ (a special case of Definition \ref{definebandmatrix}), a central problem (see \cite{cook1article},  \cite{cook2article}, \cite{tikhomirov2023pseudospectrum}) is to prove the convergence of the empirical spectral density to a deterministic limit. The limit should be the circular law as long as $d_N$ grows with $N$, which has been proved for sparse i.i.d. matrices in \cite{rudelson2019sparse}. The major technical difficulty to proving the circular law (see for example \cite{ WOS:000281425000010}) is to prove an estimate on the least singular value of $z\mathbf{1}-(d_N)^{-1/2}X_N$ for a.e. $z\in\mathbb{C}$, so that Girko's Hermitization technique \cite{girko2article} \cite{girkoarticle} can be rigorously justified. This step is highly model specific and the proof techniques for sparse i.i.d. matrices \cite{rudelson2019sparse} do not generalize to arbitrary band matrices which has much less homogeneity. Recently, \cite{jain2021circular} and \cite{tikhomirov2023pseudospectrum} introduced new methods to prove the circular law for periodic block band matrices and periodic band matrices for bandwidth $d_N\gg N^{33/34}$. Later, \cite{han2024circular}, \cite{han2025circular}, \cite{han2025circular2} pushed the validity of circular law much further and up to $d_N\geq N^{1/2+c}$ for a wide class of variance profiles, and for any $d_N$ growing to infinity when considering the open-boundary block tridiagonal model. More recently, Han \cite{han2026circular} established the circular law under $d_N\to\infty$ for several periodic band models with bounded-density entries, and under $d_N\gg\log N$ for periodic full-block matrices with real subgaussian entries, including discrete laws. Beyond these periodic model classes, for an arbitrary doubly stochastic variance profile, the case of smaller $d_N$, especially $d_N\ll\sqrt{N}$, still remains largely open.

The spectral radius for inhomogeneous random matrices is much easier to study thanks to the method of moments. Via computing the trace of a sufficiently large $\ell\sim\log N$ power of $(d_N)^{-1/2}X_N$, the spectral radius of $(d_N)^{-1/2}X_N$ can be bounded with very high probability. In \cite{benaych2020spectral}, Theorem 2.11 and Remark 2.13, a method is introduced to bound the spectral radius for inhomogeneous non-Hermitian random matrix that covers $(d_N)^{-1/2}X_N$ when $d_N\geq N^\epsilon$ for any $\epsilon>0$. In Theorem \ref{theorem1.1} we illustrate the same spectral radius bound can be achieved whenever $d_N=(\log N)^5$. We note that when $g_{xy}$ are standard (real) Gaussian variables, \cite{bandeira2016sharp}, Theorem 3.1 proved that $\|(d_N)^{-1/2}X_N\|_{op}\leq 2(1+o(1))$ with high probability whenever $d_N\gg\log N$. This result is optimal in the sparsity of $d_N$ and implies $\rho((d_N)^{-1/2}X_N)\leq2(1+o(1))$, but the operator norm bound loses a factor 2 compared with the spectral radius bound $1+o(1)$.

The problem of determining the spectral outliers of $(d_N)^{-1/2}X_N+C_N$ for some finite rank perturbation $C_N$, does not appear to be studied before in such an inhomogeneous setting, and is the main topic of the current work. On one hand, finding spectral outliers is more tractable than proving the empirical density convergence: We do not need to estimate the least singular value of $X_N$, which is a notoriously difficult problem (see for instance \cite{jain2021circular}, \cite{tikhomirov2023pseudospectrum}). We often only need a least singular value estimate for $(d_N)^{-1/2}X_N-z$ for $z$ lying outside the support of the limiting circular law spectrum. This is much easier to derive, see Theorem \ref{theorem2.3} and estimates \eqref{yzfree}, \eqref{singularsupports}.

On the other hand, it appears that all existing methods on this topic require strong homogeneity on the variance profile of $X_N$. The first work for non-Hermitian matrices \cite{tao2013outliers} made use of a moment comparison to the Ginibre ensemble, and later works either followed a similar line, or resorted to proving an isotropic law (see \cite{article1221}) whose proof does not carry over to general inhomogeneous matrices $X_N$ with $d_N=o(N)$. The method of moments, or some modifications of it, might be a promising route to study outliers of $(d_N)^{-1/2}X_N+C_N$. However, as the eigenvalues are complex, its distribution is not completely determined by the moments. The method of characteristic functions can alternatively be used to detect spectral outliers, see \cite{bordenave2021convergence}. However, for this method to work we need to prove joint convergence of traces of $(d_N)^{-1/2}X_N$ to a family of independent Gaussian variables, which is simply not the case in the current setting of inhomogeneous random matrices on regular graphs. Indeed, the characteristic function $\det(\mathbf{1}-z(d_N)^{-1/2}X_N)$ does not form a tight family of holomorphic functions for $|z|\leq 1-\epsilon$, so that the method breaks down completely. Relevant moment computations for band matrices, also called genus expansion when the entries are Gaussian distributions, can be found in \cite{dubach2021words} for the i.i.d. case and \cite{au2021finite} for the Hermitian case. These genus expansion formulas (asymptotic freeness up to first order) are only proven to be effective for $d_N\gg N^{1/3}$ and $d_N\gg N^{1/2}$ respectively, see the respective papers \cite{dubach2021words} and \cite{au2021finite}, falling short of the smaller $d_N$ case. It is not clear whether the results in \cite{dubach2021words} and \cite{au2021finite} can be extended to all polynomially growing $d_N\gg N^{O(1)}$ and we do not pursue it here. Therefore, we see that genus expansion is ineffective for $d_N=o(N^{1/3})$ and the characteristic function technique is ineffective for any $d_N=o(N)$, and completely new techniques are needed here.

We will first take a standard Hermitization procedure and consider for any $z\in\mathbb{C}$, $$
\mathcal{Y}_z:=\begin{pmatrix}
    0& X-z\mathbf{1}\\X^*-\bar{z}\mathbf{1}&0
\end{pmatrix}.
$$
 We need to show that $\sigma_{\min}(\mathcal{Y}_z)>0$ when $z$ is away from the unit disk (or ellipse $\mathcal{E}_{\rho}$), and prove an isotropic law for $\mathcal{Y}_z$ for $z$ in the same region. The inhomogeneous sparse variance pattern of $X$ makes standard techniques ineffective, and we will make use of recent advances in free probability and universality principles for highly inhomogeneous self-adjoint random matrices, which are illustrated in \cite{bandeira2023matrix} for the Gaussian case and \cite{brailovskaya2022universality} for the general case. The idea is to consider a free probability model $\mathcal{Y}_{z,\text{free}}$ defined from $\mathcal{Y}_z$, and the assumptions on $X$ in Theorem \ref{gaussianmastertheorem} (the variance profile being doubly stochastic and $\mathbb{E}[X^2]$ is a multiple of identity) make the spectral properties of $\mathcal{Y}_{z,\text{free}}$ remarkably transparent. Then universality principles in the cited papers enable us to transfer the estimates back to our random matrix $X$. Similar ideas have appeared in \cite{shou2024localization} which also partially inspires our work.

The main result of this paper can be seen as a non-Hermitian generalization of recent studies of spectral outliers for random band matrices. For periodically banded Hermitian matrices, the finite rank perturbation properties (the BBP transition introduced in \cite{peche2006largest}) are recently studied in \cite{au2023bbp}. Later \cite{bandeira2024matrix} introduced a very general framework for solving such problems in the self-adjoint setting, and established the BBP type phenomena for many other inhomogeneous random matrices with diverse patterns of variance profiles. Theorem \ref{gaussianmastertheorem} can be thought of as a non-Hermitian analogue of \cite{bandeira2024matrix}, Theorem 3.1. In contrast to the Hermitian case, we do not discuss the eigenvector overlaps associated to the outlying eigenvalues. While this might theoretically be possible to investigate, a closed form expression is too difficult to work out.

\section{Preparations: free probability and universality theorems}
\label{proof outlines}

\subsection{Free probability preliminaries}
We first recall some standard facts in free probability that can be found in \cite{nica2006lectures}.
Consider a $C^*$ probability space $(\mathcal{A},^*,\tau,\|\cdot\|)$, where $(\mathcal{A},^*,\|\cdot\|)$ is a $C^*$-algebra and $\tau$ is a faithful trace. Recall that a trace $\tau$ is called a faithful trace if $\tau(a^*a)=0$ implies $a=0$.

As $\tau$ is a faithful trace, the norm $\|\cdot\|$ is uniquely determined by (\cite{nica2006lectures}, Proposition 3.17)
\begin{equation}
    \|a\|=\lim_{k\to\infty}(\tau[(a^*a)^k])^\frac{1}{2k},\quad \forall a\in\mathcal{A}.
\end{equation}

A family of self-adjoint elements $s_1,\cdots,s_m\in\mathcal{A}$ is called a free semicircular family if for all $p\geq 1$ and $k_1,\cdots,k_p\in [m]$ we have
$$
\tau(s_{k_1}\cdots s_{k_p}) =\sum_{\pi\in\operatorname{NC}_2([p])}\prod_{\{i,j\}\in\pi }\delta_{k_i,k_j},
$$
where $\operatorname{NC}_2([p])$ is the set of non-crossing pair partitions of $[p]$.

\subsection{Universality principles}
\label{universalityprinciple}

For the Gaussian case, let $X$ be a Gaussian random matrix admitting a decomposition \eqref{gaussianmatrixform}. Although in the statement of Theorem \ref{gaussianmastertheorem} the matrix $X$ is non-Hermitian, we need to consider Hermitian matrices when stating the spectral universality theorems. This will be achieved via a standard Hermitization procedure.

Therefore, in this section we consider a Gaussian Hermitian random matrix of the following form. We are given fixed matrices $A_0,\cdots,A_n\in M_N(\mathbb{C})$, where $M_N(\mathbb{C})$ denotes the set of $N$ by $N$ square matrices with complex entries. Let $g_1,\cdots,g_n$ be standard real Gaussian variables with mean 0 and variance 1. We define a canonical Gaussian model \begin{equation}\label{gaussianmatrixformsecond}
    G:=A_0+\sum_{i=1}^n g_i A_i. 
\end{equation}
In this section we assume $G$ is Hermitian, so we require that each $A_0,A_1,\cdots,A_n$ is Hermitian as well, and we will use the notation $M_N(\mathbb{C})_{sa}$ for the set of self-adjoint matrices in $M_N(\mathbb{C})$. Then it is clear that any square Hermitian Gaussian random matrix admits a decomposition as \eqref{gaussianmatrixformsecond}.

To the Gaussian model \eqref{gaussianmatrixformsecond}
we introduce the following free probability model $G_{free}$ associated to $G$:

\begin{equation}\label{freemodelgaussian}
G_{\text{free}}:=A_0\otimes \mathbf{1}+\sum_{i=1}^n A_i\otimes s_i,
\end{equation}where $A_0,\cdots,A_n\in M_N(\mathbb{C})_{sa}$ and $s_1,\cdots,s_n$ are a free semicircular family.

Next, we consider a general Hermitian random matrix model $W$ defined as
\begin{equation}\label{generalmodels}
W:=Z_0+\sum_{i=1}^n Z_i
\end{equation} where $Z_0,\cdots,Z_n\in M_N(\mathbb{C})$, $Z_0$ is deterministic, and $Z_1,\cdots,Z_n$ are independent Hermitian random matrices with mean zero. (We note in passing that we can also assume $Z_0,\cdots,Z_n\in M_N(\mathbb{R})$ and are symmetric: the results stated below carry over without any change).

For this general matrix model $W$, we can associate it with a Gaussian model $G=(G_{ij})$ which is an $N\times N$ matrix (written in the form \eqref{gaussianmatrixformsecond}) such that $\{\operatorname{Re}G_{ij},\operatorname{Im}G_{ij}:i,j\in[N]\}$ are jointly Gaussian, that $\mathbb{E}[G]=\mathbb{E}[W]$ and that $\operatorname{Cov}(G)=\operatorname{Cov}(W)$. Then we associate to $W$ a free probability model $W_{\text{free}}$ via setting $$W_{\text{free}}:=G_{\text{free}},$$ where $G_{\text{free}}$ is the free probability model associated to the Gaussian model $G$ as defined in \eqref{freemodelgaussian}. In other words, the free probability model of $W$ is defined as the free probability model associated to its Gaussian model $G$. 

Now we define a family of matrix parameters following \cite{bandeira2023matrix}, \cite{brailovskaya2022universality}. For the Gaussian model $G$ \eqref{gaussianmatrixformsecond}, we recall that its covariance profile $\operatorname{Cov}(G)\in M_{N^2}(\mathbb{C})$ is defined in Section \ref{section247} via 
$$
\operatorname{Cov}(G)_{ij,kl}=\sum_{s=1}^n (A_s)_{ij}\overline{(A_s)_{kl}}, 
$$ and that we have defined 
$$\begin{aligned}
&v(G)^2=\|\operatorname{Cov}(G)\|=\sup_{\operatorname{Tr}|M|^2\leq 1}\sum_{s=1}^n |\operatorname{Tr}[A_sM]|^2,\\
&\sigma(G)^2:=\|\mathbb{E}[(G-\mathbb{E}G)^*(G-\mathbb{E}G)]\|\vee\|\mathbb{E}[(G-\mathbb{E}G)(G-\mathbb{E}G)^*]\|,\\
&\sigma_*(G)^2:=\sup_{\|v\|=\|w\|=1}\mathbb{E}[|\langle v,(G-\mathbb{E}G)w\rangle|^2],
\end{aligned}$$
and finally we have defined $$\tilde{v}(G)^2:=v(G)\sigma(G).$$

For the non-Gaussian model $W$ we define its covariance via
\begin{equation}
    \operatorname{Cov}(W)_{ij,kl}:=\mathbb{E}[(W-\mathbb{E}W)_{ij}\overline{(W-\mathbb{E}W)_{kl}}],
\end{equation}
and we define $v(W),\sigma(W),\sigma_*(W)$ as in the definition of the Gaussian case. We also define 
\begin{equation} R(W):=\left\|\max_{1\leq i\leq n}\|Z_i\|\right\|_\infty,\quad 
    \bar{R}(W):=\mathbb{E}\left[\max_{1\leq i\leq n}\|Z_i\|^2\right]^\frac{1}{2}.
\end{equation}

For the Gaussian model $G$, we have the following concentration and universality results:

\begin{theorem}(Resolvent convergence, Gaussian case, \cite{bandeira2023matrix}, Theorem 2.8)\label{resolventconvergencegauss} Consider the Gaussian model $G$ \eqref{gaussianmatrixformsecond}.
Let $A_0,\cdots,A_n\in M_N(\mathbb{C})_{sa}$. The matrix-valued Stieltjes transform for any $Z\in M_N(\mathbb{C})$ is defined as
$$
G(Z):=\mathbb{E}[(Z-G)^{-1}],\quad G_{free}(Z):=(\operatorname{id}\otimes\tau)[(Z\otimes\mathbf{1}-G_{free})^{-1}].
$$
    Then we have the estimate
    $$
\|G(Z)-G_{free}(Z)\|\leq\tilde{v}(G)^4\|(\operatorname{Im}Z)^{-5}\| 
    $$ which holds for any $Z\in M_N(\mathbb{C})$ such that $\operatorname{Im}Z:=\frac{1}{2i}(Z-Z^*)>0$.
\end{theorem}

For the Gaussian case we only need the following concentration estimate, which is very similar to \cite{brailovskaya2022universality}, Lemma 5.5. The proof is given in Appendix \ref{sectionappendixa}.
\begin{theorem}(Resolvent concentration, Gaussian case)\label{resolventconcentrationgauss} Consider the Gaussian model $G$ \eqref{gaussianmatrixformsecond}. Fix any two unit vectors $u,v\in\mathbb{C}^N$. Fix $z\in\mathbb{C}$ with $\operatorname{Im}z>0$. For any $x>0$ we have
    $$
\mathbb{P}\left[\left| \langle u,
(z\mathbf{1}-G)^{-1}v\rangle-\mathbb{E}\langle u,(z\mathbf{1}-G)^{-1}v\rangle\right|\geq\frac{\sqrt{2}\,\sigma_*(G)}{(\operatorname{Im}z)^2}
x\right]\leq 4e^{-x^2/2}.
    $$
\end{theorem}
Applying concentration separately to the real and imaginary parts and taking a union bound gives the stated complex-valued estimate.
For any self-adjoint operator $X$ we let $\operatorname{sp}(X)\subset\mathbb{R}$ denote the spectrum of the operator $X$. We have

\begin{theorem}\label{theorem2.3}(Spectrum concentration, Gaussian case, \cite{bandeira2023matrix} Theorem 2.1)\label{spectraluniversality}
   Consider the Gaussian model $G$ \eqref{gaussianmatrixformsecond}. For all $t>0$,
    $$
\mathbb{P}\left[ \operatorname{sp}(G)\subseteq \operatorname{sp}(G_{\text{free}})+C\{\tilde{v}(G)(\log N)^{3/4}+\sigma_*(G)t\}[-1,1]
\right]\geq 1-e^{-t^2},
    $$ for a universal constant $C>0$.
\end{theorem}

For the general case, we will need similar convergence and concentration results. We have the following universality principle:

\begin{theorem}(General entry resolvent universality, \cite{brailovskaya2022universality} Theorem 2.11) 
\label{theorem2.41} Consider the general model $W$ \eqref{generalmodels}, and let $G$ denote its associated Gaussian model as previously defined.
For every $z\in\mathbb{C}$ such that $\operatorname{Im}(z)>0$, we have
$$
\|\mathbb{E}[(z\mathbf{1}-W)^{-1}]-\mathbb{E}[(z\mathbf{1}-G)^{-1}]\|\lesssim \frac{\sigma_*(W)+\bar{R}(W)^{1/10}\sigma(W)^{9/10}}{(\operatorname{Im}z)^2}.
$$
\end{theorem}

We will also need the following concentration result on resolvent entries. The proof of the next theorem is outlined in Appendix \ref{sectionappendixa} and is very similar to the argument in \cite{brailovskaya2022universality}, Proposition 5.6.

\begin{theorem}\label{generalconcentration}(General entry resolvent concentration) Consider the general model $W$ \eqref{generalmodels}.
For every $z\in\mathbb{C}$ such that $\operatorname{Im}(z)>0$, and any two unit vectors $u,v\in\mathbb{C}^N$ with $\|u\|_{L^2}=\|v\|_{L^2}=1$, we have
$$\begin{aligned}
&\mathbb{P}[\left|\langle u,(z\mathbf{1}-W)^{-1}v\rangle-\mathbb{E}\langle u,(z\mathbf{1}-W)^{-1}v\rangle\right|> \frac{\sigma_*(W)}{(\operatorname{Im}z)^2}\sqrt{x}+\{\frac{R(W)}{(\operatorname{Im}z)^2}+\frac{R(W)^2}{(\operatorname{Im}z)^3}\}x\\&
\quad\quad +\{ \frac{R(W)^{1/2}(\mathbb{E}\|W-\mathbb{E}W\|)^{1/2}}{(\operatorname{Im}z)^2}+\frac{R(W)(\mathbb{E}\|W-\mathbb{E}W\|^2)^{1/2}}{(\operatorname{Im}z)^3}
\}\sqrt{x}
]\leq 2e^{-Cx}
\end{aligned}$$
for a universal constant $C>0$ and for any given $x\geq 0$.
\end{theorem}
Again, the original proof was for $u,v\in\mathbb{R}^N$ and we generalize it to the complex case.

\begin{theorem}(Sharp matrix concentration, \cite{brailovskaya2022universality} Theorem 2.16)\label{theorem2.61} Consider the general model $W$ \eqref{generalmodels}. Assume that we have $Z_0,\cdots,Z_n\in M_N(\mathbb{C})_{sa}$, such that $\|Z_i\|\leq R$ a.s. for each $i=1,\cdots,n$. Then we have for all $t\geq 0$, 
       $$\begin{aligned}
\mathbb{P}\Big[ \operatorname{sp}(W)\subseteq \operatorname{sp}(W_{\text{free}})+C\big\{&\widetilde v(W)(\log N)^{\frac{3}{4}}+\sigma_*(W)t^{\frac{1}{2}}\\
&+R^\frac{1}{3}\sigma(W)^{\frac{2}{3}}t^\frac{2}{3}+Rt\big\}[-1,1]\Big]\geq 1-2Ne^{-t}.
    \end{aligned}$$
    where $C>0$ is a universal constant.
\end{theorem}

\section{Proof of main results}

\label{proofofmainresults}
We first prove Lemma \ref{triviallemmas}. We claim no novelty for its proof as can be found in (4.6)-(4.8) of \cite{alt2022local}.

\begin{proof}[\proofname\ of Lemma \ref{triviallemmas}]  First assume that $\rho$ is real, then taking the real and imaginary parts of $z=x+\rho\bar{
x}$ implies $\operatorname{Re}z=\operatorname{Re}x+\rho\operatorname{Re}x$ and $\operatorname{Im}z=\operatorname{Im}x-\rho\operatorname{Im}x$. Then since $|x|\leq 1$, we have $$\frac{\operatorname{Re}(z)^2}{(1+\rho)^2}+\frac{\operatorname{Im}(z)^2}{(1-\rho)^2}\leq 1.$$ All these relations can be reversed, hence proving equivalence.

For general $\rho\in\mathbb{C}$, we have that $z=x+\rho\bar{x}$ if and only if $$e^{-i\arg(\rho)/2}z=e^{-i\arg(\rho)/2}x+|\rho|\overline{e^{-i\operatorname{arg}(\rho)/2}x}$$ for any $x\in\mathbb{C}:|x|\leq 1$. This implies the equivalence.    
\end{proof}

Before entering the proof we introduce elliptic random variables in free probability, which generalizes the semicircular and circular variables introduced by Voiculescu \cite{voiculescu1991limit}.

We shall also use some notations on operator-valued free probability theory introduced in \cite{voiculescu1995operations} and we also refer to \cite{belinschi2021outlier}, Section 3.2.1 for some basic notations.

\begin{Definition}
    In a $C^*$-non-commutative probability space $(\mathcal{A},\tau)$, an element $c$ is called a circular element if $c=\frac{1}{\sqrt{2}}s_1+\frac{i}{\sqrt{2}}s_2$ for two free semicircular variables $s_1,s_2$.

    In general, for any $\rho\in[-1,1]$ we call $c_\rho:=\sqrt{\frac{1+\rho}{2}}s_1+i\sqrt{\frac{1-\rho}{2}}s_2$ an elliptic element of parameter $\rho$ where $s_1,s_2$ are two free semicircular variables.
\end{Definition}

\subsection{No outliers in Theorem \ref{gaussianmastertheorem}}  
We start with proving the master theorem in Gaussian case, Theorem \ref{gaussianmastertheorem}.

First we need to analyze the free operator $X_{\text{free}}$ associated to $X$. After drafting this paper, we learned that similar computations were previously done in \cite{erdHos2023randomly}. From a free probability standpoint, the complex $\rho$ case is sometimes called a twisted elliptic operator, where recent developments in free probability, such as  \cite{belinschi2024brown}, could lead to much more simplified or general treatments of spectral measure than the one presented here.
But here we still offer a complete derivation for the sake of completeness.

Fix any $z\in\mathbb{C}$ and consider the Hermitization of $X$ at $z$: define
$$
\mathcal{Y}_z:=\begin{pmatrix}
    0& X-z\mathbf{1}\\{X}^*-\bar{z}\mathbf{1}&0,
\end{pmatrix} 
$$ where $X^*$ denotes the conjugate transpose of $X$, and in particular when $z=0$ we write 
$$
\mathcal{Y}=\begin{pmatrix}
    0&X\\X^*&0
\end{pmatrix}.
$$
The associated free model can be written as 
$$
\mathcal{Y}_{z,\text{free}}=\begin{pmatrix}
    0& -z\mathbf{1}\\-\bar{z}\mathbf{1}&0\end{pmatrix}+\begin{pmatrix}
    0& X_{\text{free}}\\X^*_{\text{free}}&0\end{pmatrix}
$$  where $X^*_{free}$ is the conjugate transpose of the free probability object $X_{free}$, which is naturally defined.

Also, for any $v=E+i\eta\in\mathbb{C}$ with $\eta>0,$ we define
$$
\mathcal{Y}_{z}(v)=\begin{pmatrix}
    -v \mathbf{1}& -z\mathbf{1}\\-\bar{z}\mathbf{1}&-v \mathbf{1}\end{pmatrix}+\begin{pmatrix}
    0& X_{}\\X^*&0\end{pmatrix},\quad\text{and}$$
$$
\mathcal{Y}_{z,\text{free}}(v)=\begin{pmatrix}
    -v \mathbf{1}& -z\mathbf{1}\\-\bar{z}\mathbf{1}&-v \mathbf{1}\end{pmatrix}+\begin{pmatrix}
    0& X_{\text{free}}\\X^*_{\text{free}}&0\end{pmatrix}.
$$

To compute the Stieltjes transform we need the following computation. For a self-adjoint random matrix $A=A_0+\sum_{i=1}^nA_ig_i$ where $A_i\in M_N(\mathbb{C})_{sa}$ are self adjoint and $g_i$ are independent real Gaussians of mean $0$, variance $1$, then for any $M\in M_N(\mathbb{C})_{sa}$, we have

\begin{equation}\label{stieltjesgagagrga}
\sum_{i=1}^n A_iMA_i=\mathbb{E}\left[(\sum_{i=1}^n A_ig_i)M(\sum_{i=1}^n A_ig_i)\right]=\mathbb{E}\left[(A-\mathbb{E}A)M(A-\mathbb{E}A)\right].
\end{equation}

We define the Stieltjes transform of the free model $\mathcal{Y}_{z,\text{free}}$ as follows: for any $v\in\mathbb{C}_+:=\{z=E+i\eta\in\mathbb{C}:\eta>0\}$, we define
$$
G_z(v):=(\operatorname{id}\otimes\tau)(\mathcal{Y}_{z,\text{free}}(v)^{-1}).
$$

By \cite{haagerup2005new}, equation 1.5 and the computation \eqref{stieltjesgagagrga}, the Stieltjes transform $G_z(v)$ solves the following self-consistency equation
\begin{equation}\label{matrixdysomequationsfag}
\mathbb{E}[\mathcal{Y}G_z(v)\mathcal{Y}]+G_z(v)^{-1}+\begin{pmatrix} v \mathbf{1}& z \mathbf{1}\\\bar{z}\mathbf{1}& v \mathbf{1}
\end{pmatrix}=0.\end{equation} By \cite{helton2007operator}, Theorem 2.1, for any $v\in\mathbb{C}_+$ there is a unique solution $G_z(v)$ to this equation such that $G_z(v)$ has a positive imaginary part. Recall that for a square matrix $M$, its imaginary part is defined as $\operatorname{Im}M=\frac{1}{2i}(M-M^*)$.

For the comparison below, fix $\eta>0$ and set $v=i\eta$. Let $J:=\operatorname{diag}(\mathbf{1},-\mathbf{1})$. The self-adjoint dilation satisfies $J\mathcal{Y}_{z,\text{free}}J=-\mathcal{Y}_{z,\text{free}}$, and therefore
$$
\big(\mathcal{Y}_{z,\text{free}}(i\eta)^{-1}\big)^*=-J\mathcal{Y}_{z,\text{free}}(i\eta)^{-1}J.
$$
In particular, the two off-diagonal blocks of $G_z(i\eta)$ are adjoints. In view of the regularity conditions, we look for a block-scalar solution to \eqref{matrixdysomequationsfag} at $v=i\eta$ of the form
\begin{equation}\label{definitiongzeta}
G_z(i\eta)=\begin{pmatrix}
    a(z,\eta) \mathbf{1}& b(z,\eta)\mathbf{1}\\ \overline{b(z,\eta)}\mathbf{1}&c(z,\eta)\mathbf{1}
\end{pmatrix},
\end{equation} where $a(z,\eta)$, $b(z,\eta)$ and $c(z,\eta)$ are scalar functions (also depending on $\rho$ and $D_N$, which are suppressed from notation). We verify below that such a solution exists; uniqueness of the full matrix Dyson equation then identifies it with $G_z(i\eta)$.

Substituting this block-scalar form into \eqref{matrixdysomequationsfag}, the equation transforms to

\begin{equation} \label{agageggwgewe}
\begin{pmatrix}
    c&(\rho+D_N)\overline{b}\\\overline{(\rho+D_N)}b&a
\end{pmatrix}
+
    \frac{1}{ac-|b|^2}\begin{pmatrix}
        c&-b\\-\bar{b}&a
    \end{pmatrix}+\begin{pmatrix}
        i\eta&z\\\bar{z}&i\eta
    \end{pmatrix}=0.
\end{equation}

We now construct a solution to \eqref{agageggwgewe} such that $\begin{pmatrix}
    a&b\\\bar{b}&c
\end{pmatrix}$ has a positive imaginary part. By uniqueness of the full matrix Dyson equation, this solution must give the block-scalar representation of $G_z(i\eta)$. We do not try to solve \eqref{agageggwgewe} directly, but instead construct the solution using another well-studied object in free probability.

For this we denote by $\rho_0:=\rho+D_N$ (so that $|\rho_0|\leq 1$ by definition) and set $\theta_0:=\arg\rho_0$ when $\rho_0\neq 0$ (and we set $\theta_0=0$ when $\rho_0=0$.) Note that by assumption we have $|\rho+D_N|\leq 1$. We consider an elliptic element $c_{|\rho_0|}$ multiplied by $e^{i\theta_0/2}$ and embed it into a two by two matrix (where $s_1$ and $s_2$ two free semicircular variables), for any $z\in\mathbb{C}$,
$$
\mathcal{C}_z=\begin{pmatrix}
    0&e^{i\theta_0/2}(\sqrt{\frac{1+|\rho_0|}{2}} s_1+i\sqrt{\frac{1-|\rho_0|}{2}} s_2)-z1_{\mathcal{A}}\\e^{-i\theta_0/2}(\sqrt{\frac{1+|\rho_0|}{2}} s_1-i\sqrt{\frac{1-|\rho_0|}{2}} s_2)-\bar{z}1_\mathcal{A}&0
\end{pmatrix}.
$$

We claim that the Stieltjes transform of $\mathcal{C}_z$ on the positive imaginary axis, defined by
$$\widetilde{\mathcal{G}}_z(i\eta):=(\operatorname{id}\otimes\tau
)((\mathcal{C}_z-i\eta I_2)^{-1}):=\begin{pmatrix}
    \widetilde{a}(z,\eta)&\widetilde{b}(z,\eta)\\ \overline{\widetilde{b}(z,\eta)}&\widetilde{c}(z,\eta)
\end{pmatrix},$$
satisfies \eqref{agageggwgewe} after substituting $a=\widetilde{a}$, $b=\widetilde{b}$, and $c=\widetilde{c}$. This follows by applying \cite{haagerup2005new}, equation 1.5.

The matrix $\widetilde{\mathcal{G}}_z(i\eta)$ has positive imaginary part as it is the Stieltjes transform of $\mathcal{C}_z$. Define its block-scalar lift by $\widehat{G}_z(i\eta):=\widetilde{\mathcal{G}}_z(i\eta)\otimes\mathbf{1}_N$.
The reduced equation \eqref{agageggwgewe} and the regularity conditions show that $\widehat{G}_z(i\eta)$ is a positive-imaginary-part solution of the full matrix Dyson equation at $v=i\eta$. Uniqueness therefore yields
$$
G_z(i\eta)=\widehat{G}_z(i\eta),\qquad \eta>0.
$$

Let $\operatorname{tr}_k:=k^{-1}\operatorname{Tr}$ denote the normalized trace on $M_k(\mathbb{C})$. By the Riesz representation theorem, we can find unique probability measures $\mu_1,\mu_2$ on the real line such that for any $\varphi\in C_0(\mathbb{R})$,
$$
\int_\mathbb{R}\varphi d\mu_1=(\operatorname{tr}_2\otimes\tau)\varphi(\mathcal{C}_z),
$$
$$
\int_\mathbb{R}\varphi d\mu_2=(\operatorname{tr}_{2N}\otimes\tau)\varphi(\mathcal{Y}_{z,\text{free}}).
$$
Let $g_1(v):=(\operatorname{tr}_2\otimes\tau)((\mathcal{C}_z-vI_2)^{-1})$ and  $g_2(v):=(\operatorname{tr}_{2N}\otimes\tau)((\mathcal{Y}_{z,\text{free}}(v))^{-1})$. Our previous discussion shows that $g_1(i\eta)=g_2(i\eta)$ for every $\eta>0$. Both functions are analytic on $\mathbb{C}_+$, so the identity theorem gives $g_1(v)=g_2(v)$ for every $v\in\mathbb{C}_+$. The uniqueness of Stieltjes transforms (equivalently, the Stieltjes inversion formula) therefore gives $\mu_1=\mu_2$. Since $\tau$ is faithful, we may choose $\varphi$ supported on a neighborhood of the complement of the spectrum of $\mathcal{C}_z$ (or $\mathcal{Y}_{z,\text{free}}$), and deduce that \begin{equation}\label{previousidentifications}\operatorname{sp}(\mathcal{C}_z)=\operatorname{sp}(\mathcal{Y}_{z,\text{free}}).\end{equation}

In the following we use the notation $\sigma_{\min}$ to denote the smallest singular value of an operator.

Now we determine $\operatorname{sp}(\mathcal{Y}_{z,\text{free}})$ from $\operatorname{sp}(\mathcal{C}_z)$ as the latter is the rotation of an elliptic element. Indeed, $\mathcal{C}_0$ is the self-adjoint dilation of an elliptic element $c_{|\rho_0|}$ rotated by $\theta_0/2$, and the spectral support of an elliptic element $c_{|\rho_0|}$ has been computed, see for example \cite{biane1999computation}. In particular, from \cite{biane1999computation}, equation (5.5) and the computations of $R$-transforms following it, we can conclude that $$ \text{For any } z\in\mathbb{C}\setminus\mathcal{E}_{|\rho_0|,\epsilon} \text{ and any } \epsilon>0, \text{we have that } \sigma_{\min}(c_{|\rho_0|}-z1_\mathcal{A})>0,$$ 
and in this section for a free operator $a$, we use the notion $\sigma_{\min}(a)$ to denote the minimum of the spectrum of $|a|$, the absolute value of $a$.

By definition, $\mathcal{E}_{\rho_0}$ for complex $\rho_0$ is the rotation of $\mathcal{E}_{|\rho_0|}$ by $e^{i\theta_0/2}$, so that for any $z\in\mathbb{C}\setminus\mathcal{E}_{\rho_0,\epsilon}$ and any $\epsilon>0$, we have that $\sigma_{\min}(e^{i\theta_0/2}c_{|\rho_0|}-z1_\mathcal{A})>0$. Using the fact that $\sigma_{\min}(e^{i\theta_0/2}c_{|\rho_0|} -z1_\mathcal{A})$ is Lipschitz continuous in $z$ and taking a finite covering over $\{z\in\mathbb{C}:|z|\leq 10\}\setminus\mathcal{E}_{\rho_0,\epsilon}$, we deduce that we can find constant $C_{\rho,\epsilon}>0$ depending only on $\rho,\epsilon$ such that 
$$\inf_{z\in \mathbb{C}\setminus\mathcal{E}_{\rho,\epsilon}} \sigma_{\min}(e^{i\theta_0/2}c_{|\rho_0|}-z1_{\mathcal{A}})\geq C_{\rho,\epsilon}>0.$$ In the last step we have used two reductions: first, using the naïve estimate $\|c_{|\rho_0|}\|\leq 6$, then for any $|z|>10$ the nonzero lower bound on $\sigma_{\min}$ is trivial; and second, as $D_N\to 0$, the Hausdorff distance satisfies $d_H(\mathcal{E}_{\rho},\mathcal{E}_{\rho_0})\leq|\rho-\rho_0|=|D_N|<\epsilon/2$ for all sufficiently large $N$. Consequently, $\mathbb{C}\setminus\mathcal{E}_{\rho,\epsilon}\subseteq\mathbb{C}\setminus\mathcal{E}_{\rho_0,\epsilon/2}$. Applying the preceding compactness argument with $\epsilon/2$, and using continuity in $\rho_0\to\rho$ to make the lower bound uniform for all sufficiently large $N$, justifies the claim.   

In the last paragraph, we have justified that (whenever $N$ is large) for any $z\in \mathbb{C}\setminus\mathcal{E}_{\rho,\epsilon}$ we have $[-C_{\rho,\epsilon},C_{\rho,\epsilon}]\cap \operatorname{sp}(\mathcal{C}_z)=\emptyset$. By previous identification \eqref{previousidentifications} we then have \begin{equation}\label{yzfree}[-C_{\rho,\epsilon},C_{\rho,\epsilon}]\cap \operatorname{sp}(\mathcal{Y}_{z,\text{free}})=\emptyset.\end{equation} Now we apply universality principles to show the same holds for $\operatorname{sp}(\mathcal{Y}_z)$ with high probability, justifying that $X$ has no outliers:

\begin{proof}[\proofname\ of Theorem \ref{gaussianmastertheorem}, no outliers] Now the proof of no-outliers is a simple application of Theorem \ref{spectraluniversality} together with estimate \eqref{yzfree}. We only need to consider the region $|z|\leq 10$, as the main results of \cite{bandeira2023matrix} already imply that $\|X\|_{op}\leq 6$ with very high probability. For each fixed $z\in\mathbb{C}\setminus\mathcal{E}_{\rho,\epsilon}$ we apply Theorem \ref{spectraluniversality} to deduce that, with probability at least $1-C_1N^{-1.5}$, $[-C_{\rho,\epsilon}/2,C_{\rho,\epsilon}/2]\cap\operatorname{sp}(\mathcal{Y}_z)=\emptyset,$ so that $\sigma_{\min}(X-z\mathbf{1})\geq C_{\rho,\epsilon}/2>0$ with probability at least $1-C_1N^{-1.5}$. Since $\sigma_{\min}(X-z\mathbf{1})$ is Lipschitz continuous in $z$, we can take a finite covering of $\{z\in\mathbb{C}:|z|\leq 10\}$ and take a union bound to conclude that with probability at least $1-C_1N^{-1.5}$, we have \begin{equation}\label{singularsupports}
\inf_{z\in\mathbb{C}\setminus\mathcal{E}_{\rho,\epsilon}}\sigma_{\min}(X-z\mathbf{1})\geq C_{\rho,\epsilon}/2>0.\end{equation} 
    The constant $C_1$ may change from line to line but it depends only on $\rho$ and $\epsilon$.
\end{proof}

\subsection{Isotropic laws}

In this section we determine an isotropic law, i.e. an almost sure limit of $\langle u,(X-z\mathbf{1})^{-1}v\rangle$ for unit vectors $u,v$ and $z\in\mathbb{C}\setminus\mathcal{E}_{\rho,\epsilon}$ in the setting of Theorem \ref{gaussianmastertheorem}.

In the following we only need to evaluate the Stieltjes transform $\mathcal{Y}_{z,\text{free}}(v)$ at $v=i\eta$ for any $\eta>0$, and we will pass to the limit $\eta\searrow 0$ in the end.

Recall that provided $X-z\mathbf{1}$ is invertible, we shall have \begin{equation}\label{yzoyzo}\mathcal{Y}_z(0)^{-1}=\begin{pmatrix}0&({X^*}-\bar{z}\mathbf{1})^{-1}\\(X-z\mathbf{1})^{-1}&0.
\end{pmatrix}.\end{equation}

For any $z\in\mathbb{C}\setminus\mathcal{E}_{\rho,\epsilon},$ on the event where \eqref{singularsupports} happens (which has probability at least $1-C_1N^{-1.5}$ under the assumption of Theorem \ref{gaussianmastertheorem}), we have that 
$$
\|\mathcal{Y}_z(0)^{-1}\|_{op}\leq 2/C_{\rho,\epsilon}<\infty
$$
is uniformly bounded for fixed $\epsilon$.
Then from the resolvent identity \begin{equation}\label{resolventequations}(\lambda \mathbf{1}-A)^{-1}-(\lambda \mathbf{1}-B)^{-1}=(\lambda \mathbf{1}-A)^{-1}(A-B)(\lambda I-B)^{-1}\end{equation}
we can deduce that for $\eta>0$ sufficiently small one must have
$$\begin{aligned}
\|\mathcal{Y}_z(i\eta)^{-1}\|_{op}\leq \|\mathcal{Y}_z(0)^{-1}\|_{op}+\|\mathcal{Y}_z(i\eta)^{-1}\|_{op}\eta\|\mathcal{Y}_z(0)^{-1}\|_{op},\end{aligned}$$ so that for $\eta\in[0,C_{\rho,\epsilon}/4]$, we have that on the event where \eqref{singularsupports} happens,
\begin{equation}\label{sufficientsmallwehave}\|\mathcal{Y}_z(i\eta)^{-1}\|_{op}\leq 2\|\mathcal{Y}_z(0)^{-1}\|_{op}\leq 4/C_{\rho,\epsilon}.
\end{equation} 

Similarly, for any $z\in\mathbb{C}\setminus\mathcal{E}_{\rho,\epsilon}$
we already have proved in \eqref{yzfree} that 
$$
\|\mathcal{Y}_{z,\text{free}}(0)^{-1}\|_{op}\leq 1/C_{\rho,\epsilon}<\infty.
$$
In the following we will compute $\lim_{\eta\searrow 0}(\mathcal{Y}_{z,\text{free}}(i\eta))^{-1}$ and use the universality principle in Section \ref{universalityprinciple} to compute $\mathcal{Y}_z(0)^{-1}$.

For $v=i\eta$, the solution to \eqref{agageggwgewe} satisfies $a=c$ by symmetry (this can be checked by noting that, swapping the role of $a$ and $c$, the triple $(c,b,a)$ is again a solution provided that $(a,b,c)$ is a solution, and the new solution still has positive definite imaginary part; then we resort to uniqueness of solutions in \cite{helton2007operator}). The individual coordinate equations are therefore given by

\begin{equation}\label{firstiterate}
a(1+\frac{1}{a^2-|b|^2})+i\eta=0,
\end{equation}

\begin{equation}\label{seconditerate}
\bar{b}(\rho+D_N)-\frac{b}{a^2-|b|^2}=-z.
\end{equation}

Next we show that $a=c$ is purely imaginary and has a positive imaginary part. We will prove this by showing that there exists a solution to \eqref{agageggwgewe} with purely imaginary $a$ that also satisfies the positive definite imaginary part assumption, so that by uniqueness \cite{helton2007operator} the solution must be of this form. We take $$a(z,\eta)=iV(z,\eta),\quad V(z,\eta)\in\mathbb{R}_+.$$ Then the equation simplifies to 
\begin{equation}\label{agageggwgeweagag}
V(1-\frac{1}{|V|^2+|b|^2})+\eta=0,\quad \bar{b}(\rho+D_N)+\frac{b}{|V|^2+|b|^2}+z=0.
\end{equation}
\begin{lemma}
    For any $\eta>0$, any $z\in\mathbb{C}$ and $|\rho+D_N|\leq 1$ there exists a solution $V>0$ to \eqref{agageggwgeweagag}.
\end{lemma}
\begin{proof} We may without loss of generality assume $\rho_0:=\rho+D_N$ is real (and recall $\theta_0=\arg\rho_0$), this is because for a solution $(V,b)$ to \eqref{agageggwgeweagag} with parameter $(\rho_0,z)$ replaced by  $|\rho_0|$ and $e^{-i\theta_0/2}z$, via a simple calculation we see that $(V,e^{i\theta_0/2}b)$ is a solution to \eqref{agageggwgeweagag} with parameter $\rho_0$ and $z$. 

Then in the case where $\rho_0$ is real, we can follow computations as in \cite{article1221}, Lemma C.2 to see that $V$ solves the following equation (we have changed notation $a\mapsto iV$ and $\eta\to i\eta$ in \cite{article1221}, C.2):
    $$
1-\frac{1}{(V+\eta)V}=-\frac{\operatorname{Re}(z)^2}{(\eta+(1+\rho_0)V)^2}-\frac{\operatorname{Im}(z)^2}{(\eta+(1-\rho_0)V)^2}
    $$ Setting $V\to 0+$ and setting $V\to +\infty$ on both sides of the equation shows that a solution must exist.
\end{proof}

Having checked $a$ is purely imaginary, we now derive $\eta\searrow 0$ asymptotics of $a$.
In the following we consider $z\in\mathbb{C}\setminus\mathcal{E}_{\rho,\epsilon}$. Suppose that $\lim\sup_{\eta\searrow 0}|a(z,\eta)|>0$, then on a sequence $\eta_n\searrow 0$ we have $|V|^2+|b|^2\to 1$. Together with the fact that $D_N\to 0$ as $N\to \infty$, we see that whenever $N$ is large enough, this forces \begin{equation}\label{thisforces}b+\bar{b}\rho=-z+o(1)\end{equation} on this sequence $\eta_n$. The asymptotic $|V|^2+|b|^2\to 1$ forces $|b|\leq 1+o(1)$. Combining this with \eqref{thisforces} and recalling the definition of $\mathcal{E}_{\rho,\epsilon}$, this leads to a contradiction to the fact that $z\in \mathbb{C}\setminus\mathcal{E}_{\rho,\epsilon}$

Thus for any fixed $\epsilon>0$ and any $z\in\mathbb{C}\setminus\mathcal{E}_{\rho,\epsilon}$, whenever $N$ is sufficiently large,  we must have $\lim_{\eta\searrow 0}|V(z,\eta)|=0$. Via a standard computation we deduce that $$\lim_{\eta\searrow 0}\bar{b}(z,\eta)=\begin{cases}\frac{-z+\sqrt{z^2-4\rho_0}}{2\rho_0}&\rho_0\neq 0,\\-\frac{1}{z}&\rho_0=0\end{cases}$$ where we take the square root $\sqrt{z^2-4\rho_0}$ with branch cut on the segment connecting $-2\sqrt{\rho_0}$ and $2\sqrt{\rho_0}$ in the complex plane and such that $\lim_{z\to\infty}\sqrt{z^2-4\rho_0}-z=0$. 

We have proved the following: for any fixed $\epsilon>0$, whenever $N$ is large enough, we have
\begin{equation}\label{finalagaggwergqwg}
    \lim_{\eta\searrow 0
    } a(z,\eta)=0,\quad \lim_{\eta\searrow 0}\bar{b}(z,\eta)=\begin{cases}
    \frac{-z+\sqrt{z^2-4\rho_0}}{2\rho_0}&\rho_0\neq 0,\\
    -\frac{1}{z}&\rho_0=0\end{cases}.
\end{equation}
We also note that 
$$
\sigma_*(\mathcal{Y}_z)\lesssim\sigma_*(X),\quad \widetilde{v}(\mathcal{Y}_z)\lesssim v(X)^{1/2},\quad R(\mathcal{Y}_z)\lesssim R(X).
$$

Now we are ready to prove the following isotropic law, which is the cornerstone to Theorem \ref{gaussianmastertheorem} and generalizes \cite{article1221}, Theorem 5.1.

\begin{theorem}\label{theisotropiclaws} (Isotropic law)
    In the setting of Theorem \ref{gaussianmastertheorem}, set $\delta_N=(\log \log N)^{-1}$. Then for any unit vectors $u,v\in\mathbb{C}^N$, with probability at least $1-C_1N^{-1.4}$, we have
\begin{equation}\label{isotropiclaws}
\begin{aligned}&
 \sup_{z\in\mathbb{C}\setminus\mathcal{E}_{\rho,\epsilon}}\left| \langle u,
(X-z\mathbf{1})^{-1}v\rangle -\frac{-z+\sqrt{z^2-4\rho}}{2\rho}\langle u,v\rangle
\right|\leq C_2\delta_N,\quad\rho\neq 0,
\\&    \sup_{z\in\mathbb{C}\setminus\mathcal{E}_{\rho,\epsilon}}\left| \langle u,
(X-z\mathbf{1})^{-1}v\rangle +\frac{1}{z}\langle u,v\rangle
\right|\leq C_2\delta_N,\quad \rho=0,
\end{aligned}\end{equation} where $C_1,C_2>0$ are some constants depending only on $\rho$ and $\epsilon$. Also, the same estimate holds if we replace $\delta_N$ by $(\delta_N)^p$ for any $p>0$.
\end{theorem}

\begin{proof}
We invoke Theorem \ref{resolventconvergencegauss} to deduce that for any $\eta>0$ and $z\in\mathbb{C}$,
\begin{equation}\label{expectionconvergencea}
   \|\mathbb{E}[(\mathcal{Y}_z(i\eta))^{-1}]-(\operatorname{id}\otimes\tau)\mathcal{Y}_{z,\text{free}}(i\eta)^{-1}\|_{op}\leq\frac{v(X)^2}{\eta^5}.
\end{equation} 
We then invoke Theorem \ref{resolventconcentrationgauss} to deduce that for any $\eta>0$ and $z\in\mathbb{C}$, for any two unit vectors $u,v\in\mathbb{C}^{2N}$, we have, for a sufficiently large fixed constant $C>0$ (taking $x=c\sqrt{\log N}$ in Theorem \ref{resolventconcentrationgauss} with $c>0$ sufficiently large),
\begin{equation}\label{585985585}
\mathbb{P}\left( \left| \langle u,
\mathcal{Y}_z(i\eta)^{-1}v\rangle -\mathbb{E}\langle u,\mathcal{Y}_z(i\eta)^{-1}v\rangle
\right|\geq \frac{C\sigma_*(X)\sqrt{\log N}}{\eta^2}\right)\leq 2N^{-10}.
\end{equation}

 In the following we take $$\eta=(\log\log N)^{-1},$$ then the right hand side of \eqref{expectionconvergencea} is $o(\delta_N)$.  

By our choice of $\eta$ and resolvent identity \eqref{resolventequations}, we can easily check that the function $z\mapsto \langle u, \mathcal{Y}_z(i\eta)^{-1}v\rangle$ is Lipschitz continuous in $z\in \mathbb{C}\setminus\mathcal{E}_{\rho,\epsilon}$ with a Lipschitz constant $(\log\log N)^2$, since we trivially have $\|\mathcal{Y}_z(i\eta)^{-1}\|_{op}\leq \log\log N.$

Then we set $T_N:=100\log\log N$. We can take a covering $\mathcal{N}$ of $B(0,T_N)\setminus\mathcal{E}_{\rho,\epsilon}$ such that for any $y\in B(0,T_N)\setminus\mathcal{E}_{\rho,\epsilon}$ we can find $z\in\mathcal{N}$ satisfying  $|y-z|\leq 10^{-2}\delta_N(\log\log N)^{-2}$. Then applying a union bound of the form \eqref{585985585} for all $z\in\mathcal{N}$, we deduce that \begin{equation}\label{lines838}
\mathbb{P}\left( \sup_{z\in\mathcal{N}}\left| \langle u,
\mathcal{Y}_z(i\eta)^{-1}v\rangle -\mathbb{E}\langle u,\mathcal{Y}_z(i\eta)^{-1}v\rangle
\right|\geq \frac{C\sigma_*(X)\sqrt{\log N}}{\eta^2}\right)= O( N^{-9}),
\end{equation} where we use that we can always take $|\mathcal{N}|=O(N)$ by our choice of $T_N$ and $\delta_N$. By our choice of $\eta$, the right hand side of \eqref{lines838} is again $o(\delta_N)$.

Finally, we use the $(\log\log N)^2$-Lipschitz continuity of the function $z\to \langle u,\mathcal{Y}_z(i\eta)^{-1}v\rangle$ in $z$, we deduce that 
\begin{equation}
\mathbb{P}\left( \sup_{z\in(\mathbb{C}\setminus\mathcal{E}_{\rho,\epsilon
})\cap \{z\in\mathbb{C}:|z|\leq T_N\}}\left| \langle u,
\mathcal{Y}_z(i\eta)^{-1}v\rangle -\mathbb{E}\langle u,\mathcal{Y}_z(i\eta)^{-1}v\rangle
\right|\geq \delta_N\right)= O( N^{-9}).
\end{equation}
 Since $\|\mathcal{Y}_0(i\eta)\|_{op}\leq 6$ for $\eta>0$ small with probability $1-O(N^{-1.4})$ (this can be proven via applying \cite{bandeira2023matrix}, Corollary 2.2),  we can ensure $\|{\mathcal{Y}_z(i\eta)
}^{-1}\|_{op}\leq\frac{1}{4}\delta_N$ for $|z|>T_N$. Indeed, for \(|z|>T_N\), the deterministic \(z\)-part dominates \(\mathcal{Y}_0(i\eta)\), so
\[
\|\mathcal{Y}_z(i\eta)^{-1}\|_{\rm op}
\le (|z|-\|\mathcal{Y}_0(i\eta)\|_{\rm op})^{-1}
\le \frac{1}{90\log\log N}
\le \frac14\delta_N
\]
for \(N\) large. Therefore, we remove the constraint $|z|\leq T_N$ and get,
\begin{equation}\label{strong585985585}
\mathbb{P}\left( \sup_{z\in\mathbb{C}\setminus\mathcal{E}_{\rho,\epsilon
}}\left| \langle u,
\mathcal{Y}_z(i\eta)^{-1}v\rangle -\mathbb{E}\langle u,\mathcal{Y}_z(i\eta)^{-1}v\rangle
\right|\geq \delta_N\right)= O( N^{-1.4}).
\end{equation}

Applying a similar continuity argument to \eqref{expectionconvergencea} we get 

\begin{equation}
 \sup_{z\in\mathbb{C}\setminus\mathcal{E}_{\rho,\epsilon
}}\left| \langle u,(id\otimes\tau)\mathcal{Y}_{z,\text{free}}(i\eta)^{-1}
v\rangle -\mathbb{E}\langle u,\mathcal{Y}_z(i\eta)^{-1}v\rangle
\right|\leq \delta_N.
\end{equation}

Combined with \eqref{strong585985585}, we get 
\begin{equation}\label{agagagagagagaggaaagg}
\mathbb{P}\left( \sup_{z\in\mathbb{C}\setminus\mathcal{E}_{\rho,\epsilon
}}\left| \langle u,
\mathcal{Y}_z(i\eta)^{-1}v\rangle -\langle u,(id\otimes\tau)\mathcal{Y}_{z,\text{free}}(i\eta)^{-1}
v\rangle
\right|\geq 2\delta_N\right)= O( N^{-1.4}).
\end{equation}

Now we set up a different estimate.
Consider the following event $$\Omega_N(z):=\{\|\mathcal{Y}_z(i\eta)^{-1}\|_{op}\leq 4/C_{\rho,\epsilon}\text{ for all } \eta\in[0,C_{\rho,\epsilon}/4]\}.$$ Since $z\in\mathbb{C}\setminus\mathcal{E}_{\rho,\epsilon}$, using the estimate \eqref{sufficientsmallwehave} we have $\mathbb{P}(\Omega_N(z))\geq 1-O(N^{-1.5})$. Since $z\mapsto \|\mathcal{Y}_z(i\eta)^{-1}\|_{op}$ is $(\log\log N)^2$-Lipschitz continuous, for a sufficiently large fixed constant $T$, we take a covering of $(\mathbb{C}\setminus\mathcal{E}_{\rho,\epsilon})\cap\{|z|\leq T\}$ of mesh size $10^{-2}/C_{\rho,\epsilon}(\log\log N)^{-2}$  to upgrade the bound to be uniform over $z\in \mathbb{C}\setminus\mathcal{E}_{\rho,\epsilon}$ (the bound $\|\mathcal{Y}_z(i\eta)^{-1}\|_{op}\leq 4/C_{\rho,\epsilon}$ for $|z|>T$ is trivial when $T$ is large), and conclude with the following: consider the event
$$\Omega_N:=\left\{\sup_{z\in\mathbb{C}\setminus\mathcal{E}_{\rho,\epsilon},\eta\in[0,C_{\rho,\epsilon}/4]}\|\mathcal{Y}_z(i\eta)^{-1}\|_{op}\leq 4/C_{\rho,\epsilon}\right\},$$
then $$\mathbb{P}(\Omega_N)\geq 1-O(N^{-1.4}).$$ Moreover, by the resolvent identity \eqref{resolventequations}, on the event $\Omega_N$ the mapping $\eta\mapsto \|\mathcal{Y}_z(i\eta)^{-1}\|_{op}$ is Lipschitz continuous in $\eta>0$ with a Lipschitz constant $(4/C_{\rho,\epsilon})^2$ uniformly in $z\in\mathbb{C}\setminus\mathcal{E}_{\rho,\epsilon}$. That is, we have that for $N$ sufficiently large,
\begin{equation}\label{Lipschotzcontomegan}
    \|\mathcal{Y}_z(i\eta)^{-1}-\mathcal{Y}_z(0)^{-1}\|_{op}\leq (16/C_{\rho,\epsilon}^2)\delta_N\text{ for all }z\in\mathbb{C}\setminus\mathcal{E}_{\rho,\epsilon}\quad\text{ on }\Omega_N.
\end{equation}

Finally, the asymptotic in \eqref{finalagaggwergqwg} for entries of the free operator, and the fact that the function $z\mapsto \|(id\otimes\tau)\mathcal{Y}_{z,\text{free}}(i\eta)^{-1}\|_{op}$
is also $(4/C_{\rho,\epsilon})^2$-Lipschitz continuous for $z\in\mathbb{C}\setminus\mathcal{E}_{\rho,\epsilon}$ (which follows from the fact that $\sigma_{\min}(\mathcal{Y}_{z,\text{free}})$ is bounded away from zero throughout $z\in\mathbb{C}\setminus\mathcal{E}_{\rho,\epsilon}$) imply via a continuity argument that
\begin{equation}\label{line765567765667}
\sup_{z\in\mathbb{C}\setminus\mathcal{E}_{\rho,\epsilon}} \left\|(id\otimes\tau)\mathcal{Y}_{z,\text{free}}(i\eta)^{-1}
-\begin{pmatrix}
    0&b(z,0+)\mathbf{1}\\\bar{b}(z,0+)\mathbf{1}&0
\end{pmatrix}\right\|_{op}\leq(16/C_{\rho,\epsilon}^2)\delta_N,
\end{equation} where we again take the choice $\eta=(\log \log N)^{-1},$ and the expression $\bar{b}(z,0+)$ is given in \eqref{finalagaggwergqwg}.

We can now conclude the proof via combining \eqref{agagagagagagaggaaagg}, \eqref{Lipschotzcontomegan} and \eqref{line765567765667}. Combining these estimates we have: for some constants $C_1,C_2>0$ depending on $\rho,\epsilon$,
\begin{equation}\label{finalfinalfinalsss}
\mathbb{P}\left( \sup_{z\in\mathbb{C}\setminus\mathcal{E}_{\rho,\epsilon
}}\left| \langle u,
\mathcal{Y}_z(0)^{-1}v\rangle -\langle u,\begin{pmatrix}
    0&b(z,0+)\mathbf{1}\\\bar{b}(z,0+)\mathbf{1}&0
\end{pmatrix}
v\rangle
\right|\geq C_2\delta_N\right)\leq C_1N^{-1.4}.
\end{equation}

Recall that $\mathcal{Y}_z(0)^{-1}$ has the expression \eqref{yzoyzo}. Thus, for any unit vector $u,v\in\mathbb{C}^N$, setting $(0,u)$ and $(v,0)$ to be the two unit vectors in $\mathbb{C}^{2N}$ in the previous estimate, we derive the (final) isotropic estimate for the original non-Hermitian matrix $X$: 
\begin{equation}\begin{aligned}&
\mathbb{P}\left( \sup_{z\in\mathbb{C}\setminus\mathcal{E}_{\rho,\epsilon
}}\left| \langle u,
(X-zI)^{-1}v\rangle -\frac{-z+\sqrt{z^2-4\rho}}{2\rho}\langle u,
v\rangle
\right|\geq C_2\delta_N\right)\leq C_1 N^{-1.4},\text{ for } \rho\neq 0
\\&
\mathbb{P}\left( \sup_{z\in\mathbb{C}\setminus\mathcal{E}_{\rho,\epsilon
}}\left| \langle u,
(X-zI)^{-1}v\rangle +\frac{1}{z}\langle u,
v\rangle
\right|\geq C_2\delta_N\right) \leq C_1 N^{-1.4},\text{ for } \rho= 0.\end{aligned}
\end{equation}
In the last step we have used  \eqref{finalagaggwergqwg} and the assumption that $\rho=\rho_0-D_N$ and that  $\lim_{N\to\infty}(\log N)D_N=0.$
Going through the above proof, we see that exactly the same argument works if we replace $\delta_N$ by $(\delta_N)^p$ and take $\eta=(\log\log N)^{-p}$ for any $p>0$.
\end{proof}

\subsection{Determination of outliers}
Finally we complete the determination of outlying eigenvalues, which completes the proof of Theorem \ref{gaussianmastertheorem}.

For this purpose, we take a finite-rank factorization for the deterministic matrix $C_N$: 
$$C_N=A_NB_N$$
where $A_N$ is some $N\times k$ matrix and $B_N$ is some $k\times N$ matrix, both of which have bounded operator norm.

Following \cite{tao2013outliers}, Lemma 2.1, we have the following eigenvalue criterion: a complex $z\in\mathbb{C}$ is an eigenvalue of $X+C_N$ but is not an eigenvalue of $X$ if and only if 
\begin{equation}
\det(1+B_N(X-z\mathbf{1})^{-1}A_N)=0.
\end{equation}
The proof is a standard linear algebra exercise using the matrix identity 
\begin{equation}\label{matrixidentity}\det(1+AB)=\det(1+BA)\end{equation}
for any $N\times k$ matrix $A$ and $k\times N$ matrix $B$.
Now we conclude the proof of Theorem \ref{gaussianmastertheorem}.

\begin{proof}[\proofname\ of Theorem \ref{gaussianmastertheorem}, determination of outliers]

Define $$f(z)=\det(1+B_N(X-z\mathbf{1})^{-1}A_N)$$
and 
$$\begin{aligned}&g(z)=\det(1+B_N\frac{-z+\sqrt{z^2-4\rho}}{2\rho}A_N),\quad \rho\neq 0,\\&g(z)=\det(1+B_N(-\frac{1}{z})A_N),\quad \rho=0.\end{aligned}$$

By a standard exercise in complex analysis, the function $z\mapsto \frac{-z+\sqrt{z^2-4\rho}}{2\rho}$ bijectively maps $\mathbb{C}\setminus\mathcal{E}_\rho$ onto the open unit disk of radius 1 centered at the origin, when $\rho\neq 0$. The same holds for $z\to-\frac{1}{z}$ when $\rho=0$.

Then using the matrix identity \eqref{matrixidentity} we deduce that the roots of $g(z)$ on $\mathbb{C}\setminus\mathcal{E}_\rho$ are precisely $\lambda+\rho/\lambda$, where $\lambda$ ranges over the eigenvalues of $C_N$ with $|\lambda|>1$, counted with algebraic multiplicity. In particular, the separation assumption in the theorem implies that the roots in $\mathbb{C}\setminus\mathcal{E}_{\rho,2\epsilon}$ are precisely the $j$ predicted locations listed there.
 
Let $K\geq 1$ be a fixed upper bound for $\operatorname{rank}(C_N)$, set
\[
r_N:=(\delta_N)^{3/4},\qquad
Z_N:=\left\{\lambda_i(C_N)+\frac{\rho}{\lambda_i(C_N)}:1\leq i\leq j\right\},
\]
and regard $Z_N$ as a multiset with algebraic multiplicities. Applying Theorem \ref{theisotropiclaws} with $\delta_N$ replaced by $(\delta_N)^{K+1}$ to the finitely many row-column pairs determined by $B_N$ and $A_N$, and using that the determinant is Lipschitz on bounded subsets of finite-dimensional matrix spaces, we obtain, on an event with probability at least $1-C_1N^{-1.4}$,
\[
\sup_{z\in\mathbb{C}\setminus\mathcal{E}_{\rho,\epsilon}}|f(z)-g(z)|
\leq C_2(\delta_N)^{K+1}.
\]
Intersecting with the no-outlier event for $X$ proved above, we may also assume that $f$ is analytic on $\mathbb{C}\setminus\mathcal{E}_{\rho,\epsilon}$.

Let $m_\rho(z)$ denote the scalar function in the definition of $g$. By \eqref{matrixidentity},
$g(z)=\prod_{\lambda\neq0}(1+m_\rho(z)\lambda)$, where the product runs over the nonzero eigenvalues of $C_N$, repeated according to algebraic multiplicity, and has at most $K$ factors. The explicit bijection above, the uniform bound on $\|C_N\|_{op}$, and the separation assumption in the theorem imply that the factors whose zeros do not belong to $Z_N$ are uniformly bounded away from zero on $\mathbb{C}\setminus\mathcal{E}_{\rho,2\epsilon}$, while every remaining factor is bounded below by $c\min\{1,|z-\zeta|\}$, where $\zeta$ is its zero. Consequently, for some $c>0$ independent of $N$,
\[
|g(z)|\geq c\min\{1,\operatorname{dist}(z,Z_N)\}^{K},
\qquad z\in\mathbb{C}\setminus\mathcal{E}_{\rho,2\epsilon},
\]
where distance is taken to the support of $Z_N$ and $\operatorname{dist}(z,\varnothing):=+\infty$.

Choose $s_N\in[r_N,2r_N]$ so that the boundary of
$U_N:=\{z:\operatorname{dist}(z,Z_N)<s_N\}$ is a finite union of piecewise smooth curves. For all sufficiently large $N$, $\overline{U_N}\subset\mathbb{C}\setminus\mathcal{E}_{\rho,2\epsilon}$. On $\partial U_N$ we have $\operatorname{dist}(z,Z_N)=s_N\geq r_N$, and hence
\[
|g(z)|\geq cr_N^K=c(\delta_N)^{3K/4}>
C_2(\delta_N)^{K+1}\geq |f(z)-g(z)|.
\]
Rouché's theorem on each connected component of $U_N$, with all boundary components included, shows that $f$ and $g$ have the same number of zeros there, counted with multiplicity. Each component contains at most $K$ elements of $Z_N$ and has diameter at most $4Kr_N$, so its zeros can be relabeled to be $O(r_N)$-close to the corresponding elements of $Z_N$.

If $z\in(\mathbb{C}\setminus\mathcal{E}_{\rho,2\epsilon})\setminus U_N$, then $\operatorname{dist}(z,Z_N)\geq s_N\geq r_N$, and the same bounds give $|f(z)-g(z)|<|g(z)|$, so $f(z)\neq0$. Finally,
$\det(X+C_N-z\mathbf{1})=\det(X-z\mathbf{1})f(z)$ identifies these zero multiplicities with the algebraic multiplicities of the outlying eigenvalues. Thus there are precisely $j$ such eigenvalues and, after relabeling, their distance from the predicted locations is
$O(r_N)=O((\log\log N)^{-3/4})$. This completes the proof of Theorem \ref{gaussianmastertheorem}.
\end{proof}

\subsection{The non-Gaussian case} The proof of the master theorem \ref{nongaussianmastertheorem} in the non-Gaussian case is a simple adaptation of the proof in the Gaussian case (although we do not merge the proofs as they use different theorems).

\begin{proof}[\proofname\ of Theorem \ref{nongaussianmastertheorem}] We follow exactly the steps in the proof of Theorem \ref{gaussianmastertheorem}, where we have the same free probability object $\mathcal{Y}_{z,\text{free}}$ having the same Stieltjes transform. It suffices to use Theorem \ref{theorem2.41}, \ref{generalconcentration} and \ref{theorem2.61} combined with (or in place of) Theorem \ref{resolventconvergencegauss}, \ref{resolventconcentrationgauss} and \ref{spectraluniversality} and all other steps are not changed: the former group of theorems provide the necessary concentration and comparison estimates for the non-Gaussian model and compares the non-Gaussian model to the associated Gaussian model $G$ or the free model $\mathcal{Y}_{z,\text{free}}$; whereas the latter group of theorems further compares the Gaussian model to the free model. To verify that parameters satisfy the claimed asymptotics, first note that, by definition,
\[
\bar R(X)=\Bigl(\mathbb E\max_i\|Z_i\|^2\Bigr)^{1/2}
\le \Bigl\|\max_i\|Z_i\|\Bigr\|_\infty=R(X).
\]
Moreover, for the Hermitized matrix \(\mathcal{Y}_z\), the corresponding block-size parameter
satisfies
\[
R(\mathcal{Y}_z)\lesssim R(X),\qquad \bar R(\mathcal{Y}_z)\lesssim \bar R(X)\le R(X),
\]
and similarly
\[
\sigma_*(\mathcal{Y}_z)\lesssim \sigma_*(X),\qquad
\widetilde {v}(\mathcal{Y}_z)\lesssim v(X)^{1/2},\qquad
\sigma(\mathcal{Y}_z)=O(1).
\]
Thus the resolvent comparison error in Theorem \ref{theorem2.41} is bounded by
\[
\frac{\sigma_*(X)+R(X)^{1/10}}{\eta^2},
\]
up to constants depending only on \(\rho,\varepsilon\), after setting
\(\eta=(\log\log N)^{-p}\). This is \(o((\log\log N)^{-p'})\) for every fixed \(p'>0\)
under the assumption \(R(X)\ll(\log N)^{-2}\) and
\(\sigma_*(X)\le(\log N)^{-3/2}\).

The concentration estimate Theorem \ref{generalconcentration} and the spectral concentration estimate
Theorem \ref{theorem2.61} use the stronger almost-sure parameter \(R(\mathcal{Y}_z)\), which is also
\(O(R(X))\). Taking \(t=4\log N\) in Theorem \ref{theorem2.61}, its failure probability is
\(O(N^{-3})\). By the preceding bounds and the assumptions of Theorem
\ref{nongaussianmastertheorem}, the four terms in its spectral error are respectively
\(O((\log N)^{-1/4})\), \(O((\log N)^{-1})\), \(o(1)\), and
\(o((\log N)^{-1})\), so the error is \(o(1)\).
Therefore the same net, continuity, and Rouché arguments used in the
proof of Theorem \ref{gaussianmastertheorem} apply without further changes. 
\end{proof}

\section{Proof of Applications}
Before proving 
Theorem \ref{theorem1.1} and \ref{TheEllipticTheorem}
via Theorem \ref{gaussianmastertheorem}, we encode in the following how we verify that a specific random matrix ensemble satisfies the assumptions of Theorem \ref{gaussianmastertheorem}.

\begin{lemma}[A matrix-unit block computation]\label{lem:block-parameter}
Let the matrix $X$ have the following decomposition into independent blocks $Z_\alpha$,
\[
X=\sum_{\alpha\in\mathcal A} Z_\alpha,\qquad
Z_\alpha=\sum_{r=1}^{q_\alpha} \xi_{\alpha r} a_{\alpha r}E_{i_{\alpha r}j_{\alpha r}},
\]
where the matrices \(Z_\alpha\) are mutually independent and have zero mean. Here $a_{\alpha r}\in\mathbb{C}$, $\xi_{\alpha r}$ are complex-valued random variables and $E_{i_{\alpha r}j_{\alpha r}}$ are the matrix units. Assume that the
matrix units \(E_{i_{\alpha r}j_{\alpha r}}\) appearing above are all distinct, that for some $\kappa_N>0$,
\[
\max_{\alpha,r}|a_{\alpha r}|\le \kappa_N,
\]
and that the block covariance and pseudo-covariance matrices satisfy
\[
\left\|\mathbb E\,\xi_\alpha\xi_\alpha^*\right\|\le L,
\qquad
\left\|\mathbb E\,\xi_\alpha\xi_\alpha^t\right\|\le L
\]
for every \(\alpha\), where \(L=O(1)\) and where \(\xi_\alpha=(\xi_{\alpha 1},\ldots,\xi_{\alpha q_\alpha})^t\). If moreover we are in the doubly stochastic case
\[
\mathbb E[XX^*]=\mathbb E[X^*X]=I,
\]
then we have the following parameter upper bounds
\begin{equation}
\sigma(X)=1,\qquad
\sigma_*(X)\le L^{1/2}\kappa_N,\qquad
v(X)\le L^{1/2}\kappa_N,\qquad
\widetilde v(X)\le L^{1/4}\kappa_N^{1/2}.
\end{equation}
In the non-Gaussian case, we have
\begin{equation}
R(X)\le \max_{\alpha\in\mathcal A}\|Z_\alpha\|_\infty.
\label{4.2}
\end{equation}
In particular, if each \(Z_\alpha\) is supported on a matching of matrix units (so that their row indices are distinct and their column indices are distinct), then
\begin{equation}
R(X)\le
\kappa_N
\max_{\alpha\in\mathcal A}
\left\|\max_{1\le r\le q_\alpha}|\xi_{\alpha r}|\right\|_\infty .
\label{4.3}
\end{equation}
\end{lemma}

\begin{proof}
For unit vectors \(u,w\), denote by
\[
b_{\alpha r}:=a_{\alpha r}\overline{u}_{i_{\alpha r}}w_{j_{\alpha r}}.
\]
Then we compute the variance via
\[
\mathbb E|\langle u,Xw\rangle|^2
=
\sum_{\alpha\in\mathcal A}
\mathbb E\left|\sum_{r=1}^{q_\alpha}\xi_{\alpha r}b_{\alpha r}\right|^2
\le
L\sum_{\alpha,r}|b_{\alpha r}|^2.
\]
Since the matrix units are distinct,
\[
\sum_{\alpha,r}|b_{\alpha r}|^2
\le
\kappa_N^2
\sum_{\alpha,r}|u_{i_{\alpha r}}|^2|w_{j_{\alpha r}}|^2
\le
\kappa_N^2
\sum_{i,j}|u_i|^2|w_j|^2
=
\kappa_N^2.
\]
Thus
\[
\sigma_*(X)^2\le L\kappa_N^2.
\]

Similarly, for any matrix \(M\) with \(\|M\|_{\mathrm{HS}}\le1\),
\[
\operatorname{Tr}(XM)
=
\sum_{\alpha,r}\xi_{\alpha r}a_{\alpha r}(M)_{j_{\alpha r}i_{\alpha r}},
\]
and the same computation gives
\[
\mathbb E|\operatorname{Tr}(XM)|^2
\le
L\kappa_N^2\sum_{i,j}|M_{ji}|^2
\le
L\kappa_N^2.
\]
Equivalently, \(v(X)^2\le L\kappa_N^2\). Since
\[
\mathbb E[XX^*]=\mathbb E[X^*X]=I,
\]
we have \(\sigma(X)=1\). Therefore
\[
\widetilde v(X)^2=v(X)\sigma(X)\le L^{1/2}\kappa_N,
\]
which gives
\[
\widetilde v(X)\le L^{1/4}\kappa_N^{1/2}.
\]
Finally, the estimate for \(R(X)\) exactly follows from its definition, being the largest one among the essential
supremum norm of each independent block $Z_\alpha$. 
\end{proof}

In Theorem \ref{TheEllipticTheorem}, each independent block corresponding to an unordered edge
\(\{x,y\}\), \(x<y\) is
\[
Z_{\{x,y\}}
=
d_N^{-1/2}\bigl((X_N)_{xy}E_{xy}+(X_N)_{yx}E_{yx}\bigr).
\]
The two-dimensional covariance and pseudo-covariance matrices of
\(((X_N)_{xy},(X_N)_{yx})\) have operator norm bounded by an absolute constant depending
only on \(|\rho|\le1\). Thus, denoting $H:=(d_N)^{-1/2}X_N$, Lemma \ref{lem:block-parameter} shows that
\begin{equation}\label{4.44}
\sigma_*(H)\le C d_N^{-1/2},\qquad
v(H)\le C d_N^{-1/2},\qquad
\widetilde v(H)\le C d_N^{-1/4},
\end{equation}
and
\begin{equation}\label{4.5}
R(H)\le
d_N^{-1/2}
\max_{\{x,y\}\in E_N,\ x<y}
\left\|
\max\{|(X_N)_{xy}|, |(X_N)_{yx}|\}
\right\|_\infty .
\end{equation}

For Theorem \ref{theorem1.1}, the same computation is indeed simpler, with one-dimensional blocks $(d_N)^{-1/2}(X_N)_{xy}E_{xy}$ indexed by directed edges.

For the periodic linearization in Theorem \ref{outliersofproducts}, the same lemma applies with
\[
\kappa_N=\max_{1\le i\le m}d_{N,i}^{-1/2}=d_{\min,N}^{-1/2}.
\]
Since \(m\) is fixed, the constant \(L\) remains \(O_m(1)\), and hence
\begin{equation}
\sigma_*(\mathcal X_N)\le C_m d_{\min,N}^{-1/2},\qquad
v(\mathcal X_N)\le C_m d_{\min,N}^{-1/2},\qquad
\widetilde v(\mathcal X_N)\le C_m d_{\min,N}^{-1/4}.
\label{4.6}
\end{equation}
Moreover,
\begin{equation}
R(\mathcal X_N)
\le
\max_{1\le i\le m} d_{N,i}^{-1/2}K_{N,i},
\label{4.7}
\end{equation}
where \(K_{N,i}\) is the corresponding essential-supremum entry/block bound for the
\(i\)-th factor.

\subsection{Truncation step}

We can now prove Theorem \ref{theorem1.1} and \ref{TheEllipticTheorem} via a truncation argument.

\begin{proof}[\proofname\ of Theorem \ref{theorem1.1}] Let $H:=(d_N)^{-1/2}X_N$ where $X_N$ is specified in Definition \ref{definebandmatrix}. By definition we verify that $$\mathbb{E}[HH^*]=\mathbb{E}[H^*H]=\mathbf{1}$$ and since no self-loops $(x,x)\in E_N$ exist, we simply have  $\mathbb{E}[H^2]=\mathbf{0}.$

    Therefore the claim of Theorem \ref{theorem1.1} for Gaussian entries (case (1)) follows from Theorem \ref{gaussianmastertheorem} with $\rho=0$, and the claim of Theorem \ref{theorem1.1} for bounded entries (case (2)) follows from Theorem \ref{nongaussianmastertheorem} with $\rho=0$. Note that the assumption $d_N\gg (\log N)^3$ was made to justify \eqref{smallness} via \eqref{4.44}, so that we can use Theorem \ref{gaussianmastertheorem}. After realizing the matrices $X_N$ on a common probability space, the almost sure statement follows from an application of Borel-Cantelli lemma. This completes the proof in cases (1) and (2).

It remains to prove Theorem \ref{theorem1.1}  case (3), where entries are i.i.d. with finite $p$-th moment. We begin with a moment estimate similar to \cite{brailovskaya2022universality}, Corollary 3.32: 
$$
\mathbb{E}\left[\sup_{(x,y)\in E_N}|(d_N)^{-1/2}g_{xy}|^2\right]\leq\frac{1}{d_N}\mathbb{E}\left[\max_{(x,y)\in E_N}|g_{xy}|^p\right]^\frac{2}{p}\leq\frac{(Nd_N)^\frac{2}{p}}{d_N}\mathbb{E}[|g_{xy}|^p]^\frac{2}{p}.
$$ By our assumption on $d_N$ we have $(d_N)^{-1}(Nd_N)^\frac{2}{p}\ll (\log N)^{-4}$. Therefore we conclude that we can find a deterministic sequence $a_N\searrow 0$ such that with probability $1-o(1)$,
\begin{equation}\label{event8888}(\log N)^2\sup_{(x,y)\in E_N}(d_N)^{-1/2}|g_{xy}|\leq  a_N.\end{equation} For regularity reasons we also assume that $a_N\geq (\log N)^{-\frac{4}{p-2}}$.

  Then we upgrade \eqref{event8888} to an almost sure convergence. We use \cite{brailovskaya2022universality}, Lemma 9.22 and the fact that $\mathbb{E}[|g_{xy}|^p]<\infty$ to deduce that 
   $$
\lim_{N\to\infty}\frac{1}{(Nd_N)^\frac{1}{p}}\sup_{(x,y)\in E_N}|g_{xy}|=0
   \quad a.s.,$$ then using our assumptions on $d_N$ and \cite{brailovskaya2022universality}, Lemma 9.21, we deduce that we can find deterministic sequence $a_N\searrow 0$, such that
   \begin{equation}\label{eventualalmostsure}
\sup_{(x,y)\in E_N}(d_N)^{-1/2}|g_{xy}|\leq (\log N)^{-2}a_N\text{ eventually } a.s.
   \end{equation} 

   Replacing $a_N$ by $\max\{a_N,(\log N)^{-\frac{4}{p-2}}\}$, we may retain the same lower bound as above.

   Now we define a truncated matrix $\widetilde{X}_N$ on $G_N$ such that for each edge $(x,y)\in E_N$, the entry $(\widetilde{X}_N)_{(x,y)}$ is $g_{xy}1_{|g_{xy}|\leq a_N(d_N)^{1/2}(\log N)^{-2}}$. We let 
   $$
\widetilde{g}_{xy}:=g_{xy}1_{|g_{xy}|\leq a_N(d_N)^{1/2}(\log N)^{-2}},
   $$ since $g_{xy}$ is symmetric we have $\mathbb{E}[\widetilde{g}_{xy}]=0$. We also compute that  $$V_N:=\mathbb{E}[|\widetilde{g}_{xy}|^2]=1-O(N^{-1}) .$$ This is because by assumption on $d_N$ we have $a_N(d_N)^{1/2}(\log N)^{-2}\geq N^\frac{1}{p-2}$, so that 
   $$\mathbb{E}[|g_{xy}|^2]-\mathbb{E}[|\widetilde{g}_{xy}|^2]=\mathbb{E}[|g_{xy}|^21_{|g_{xy}|\geq a_N(d_N)^{1/2}(\log N)^{-2}}]\leq N^{-1}\mathbb{E}[|g_{xy}|^p].
   $$

   Therefore we may instead consider $$\widetilde{H}:=(d_NV_N)^{-1/2}\widetilde{X}_N,$$ or equivalently it has the following entry-wise truncation
   $$
   {\widetilde{H}}_{ij}=(V_N)^{-1/2}H_{ij}\mathbf{1}_{|H_{ij}|\leq a_N(\log N)^{-2}}
   ,$$and now this matrix satisfies the assumptions of Theorem \ref{theorem1.1}, case (2) (we verify that the entries are bounded and $\mathbb{E}[\widetilde{H}\widetilde{H}^*]=\mathbb{E}[\widetilde{H}^*\widetilde{H}]=\mathbf{1}$, $\mathbb{E}[\widetilde{H}^2]=\mathbf{0}$), so the claim on absence of outliers (part (1)) and on convergence of outlying eigenvalues (part (2)) hold for $\widetilde{H}$. Moreover, $X_N=\widetilde{X}_N$ eventually almost surely thanks to \eqref{eventualalmostsure}, so that $\widetilde{H}=(V_N)^{-1/2}H$ eventually almost surely. Since $V_N=1-o(1)$, this factor can approximately be neglected in the limit. Thus an elementary exercise enables us to transfer the finite rank perturbation results on $\widetilde{H}$ to that of $H$. (We give the details for a similar step in the last paragraph before Section \ref{subsection4.2}. The details are presented in the final paragraph of this subsection.)
\end{proof}

For banded elliptic random matrices, a very similar argument can be used to prove Theorem \ref{TheEllipticTheorem}.

\begin{proof}[\proofname\ of Theorem \ref{TheEllipticTheorem}] 
    Let $H:=(d_N)^{-1/2}X_N$. Then by definition we have 
    $$
\mathbb{E}[HH^*]=\mathbb{E}[H^*H]=\mathbf{1},\quad \mathbb{E}[H^2]=\rho \mathbf{1}.
    $$

    In case (1) when entries are Gaussian, the claim is implied by Theorem \ref{gaussianmastertheorem}. In case (2) when entries are bounded, the claim is implied by Theorem \ref{nongaussianmastertheorem}.

    It remains to consider case (3) when $(g_1,g_2)$ have finite $p$-th moment. As in the previous proof, applying \cite{brailovskaya2022universality}, Lemma 9.22 to the following two families of random variables $\{(X_N)_{(x,y)},x< y,(x,y)\in E_N\}$, $\{(X_N)_{(x,y)},y< x,(x,y)\in E_N\}$  (each group consists of i.i.d. random variables) and using the fact that $\mathbb{E}[|g_i|^p]<\infty$ for $i=1,2$, we conclude that    $$
\lim_{N\to\infty}\frac{1}{(Nd_N)^\frac{1}{p}}\sup_{(x,y)\in E_N}\max\left(|(X_N)_{(x,y)}|,|(X_N)_{(y,x)}|\right)=0
   \quad a.s.,$$ then using our assumptions on $d_N$ and \cite{brailovskaya2022universality}, Lemma 9.21, we deduce that we can find a deterministic sequence $a_N\searrow 0$, such that
   \begin{equation}\label{ASSTATEMENT}
\sup_{(x,y)\in E_N}(d_N)^{-1/2}\max\left(|(X_N)_{(x,y)}|,|(X_N)_{(y,x)}|\right)\leq (\log N)^{-2}a_N\text{ eventually } a.s.
   \end{equation}
 We again assume that $a_N\geq (\log N)^{-\frac{4}{p-2}}$, as this can be achieved by replacing $a_N$ by $\max(a_N, (\log N)^{-\frac{4}{p-2}})$.

   Now we define a truncated matrix $\widetilde{X}_N$ on $G_N$ such that for each edge $(x,y)\in E_N$, the entry $(\widetilde{X}_N)_{(x,y)}$ is $(X_N)_{(x,y)}1_{|(X_N)_{(x,y)}|\leq a_N(d_N)^{1/2}(\log N)^{-2}}$. Then $\widetilde{X}_N=X_N$ eventually almost surely. We now verify that this truncation approximately preserves the variance and pseudo-variance. Set
\[
t_N:=a_N(d_N)^{1/2}(\log N)^{-2}
\]
and define the truncated pair on the original probability space by
\[
(\widetilde{g}_1,\widetilde{g}_2):=
(g_1 1_{|g_1|\leq t_N},g_2 1_{|g_2|\leq t_N}).
\]
Since $g_1,g_2$ are symmetric, $\widetilde{g}_1,\widetilde{g}_2$ still have mean zero. As in the previous proof, applying Markov's inequality, the fact that $g_1\stackrel{\text{law}}{=}g_2$, and the finite $p$-th moment of $g_i$, we check that 
$$
\mathbb{E}[|\widetilde{g_1}|^2]=\mathbb{E}[|\widetilde{g_2}|^2]=1-O(N^{-1}).
$$
Moreover, H\"older's inequality followed by Markov's inequality gives
\[
\begin{aligned}
\left|\mathbb{E}[\widetilde{g}_1\widetilde{g}_2]-\rho\right|
&\leq \mathbb{E}\left[|g_1g_2|1_{\{|g_1|>t_N\}\cup\{|g_2|>t_N\}}\right]\\
&\leq \left(\mathbb{E}|g_1g_2|^{p/2}\right)^{2/p}
\mathbb{P}\left(|g_1|>t_N\text{ or }|g_2|>t_N\right)^{1-2/p}\\
&\leq C t_N^{-(p-2)}=O(N^{-1}),
\end{aligned}
\]
where the last estimate uses the lower bound on $a_N$ and the assumed growth of $d_N$.

   Set $V_N:=\mathbb{E}[|\widetilde{g}_1|^2]=1-O(N^{-1})$ and
   $D_N:=\mathbb{E}[\widetilde{g}_1\widetilde{g}_2]/V_N-\rho$. Then $D_N=O(N^{-1})$, and the Cauchy--Schwarz inequality gives $|\rho+D_N|=|\mathbb{E}[\widetilde{g}_1\widetilde{g}_2]|/V_N\leq 1$. Consequently, $\widetilde{H}:=(d_NV_N)^{-1/2}\widetilde{X}_N$ satisfies
   $$\mathbb{E}[\widetilde{H}]=0,\quad\mathbb{E}[\widetilde{H}\widetilde{H}^*]=\mathbb{E}[\widetilde{H}^*\widetilde{H}]=\mathbf{1},\quad\mathbb{E}[\widetilde{H}^2]=(\rho+D_N)\mathbf{1}.
   $$
   Together with \eqref{4.44}, \eqref{4.5}, and the truncation bound above, the fact that $(\log N)D_N\to 0$ shows that all the assumptions in Theorem \ref{nongaussianmastertheorem} are satisfied. By \eqref{ASSTATEMENT} we have $$H=(V_N)^{1/2}\widetilde{H}\text{ eventually almost surely}.$$ Since we have proven no outlier and finite rank perturbation theorems (i.e. the first and second conclusions of Theorem \ref{TheEllipticTheorem}) for the matrix $\widetilde{H}$, as we have $V_N\to 1$, this implies the same limit of outlying eigenvalues for $H$. The details are given as follows:
   
   We outline the technical steps for the outlying eigenvalues, and the steps for the no-outliers are even easier.
   From Theorem \ref{nongaussianmastertheorem}, we determine that outlying eigenvalues of $\widetilde{H}+(V_N)^{-\frac{1}{2}}C_N$ are asymptotically close to outlying eigenvalues of $(V_N)^{-\frac{1}{2}}C_N$ under the mapping $x\mapsto x+\frac{\rho}{x}$, in the sense that the distance of outlying eigenvalues to the image under mappings of outlying eigenvalues of $(V_N)^{-\frac{1}{2}}C_N$ converges to 0 with probability $1-C_1N^{-1.4}$, which implies almost sure convergence via Borel-Cantelli. Meanwhile, outlying eigenvalues of $(V_N)^{-\frac{1}{2}}C_N$ converge to outlying eigenvalues of $C_N$ as $V_N\to 1$. Since the mapping $x\mapsto x+\frac{\rho}{x}$ is Lipschitz in the given region, we deduce that outlying eigenvalues of $\widetilde{H}+(V_N)^{-\frac{1}{2}}C_N$ are asymptotically close (as $N\to\infty$) to outlying eigenvalues of $C_N$ mapped via $x\mapsto x+\frac{\rho}{x}$. Since $H+C_N=(V_N)^{\frac{1}{2}}(\widetilde{H}+(V_N)^{-\frac{1}{2}}C_N)$ eventually almost surely and $V_N\to 1$, we have verified that outlying eigenvalues of $H+C_N$ are asymptotically close to outlying eigenvalues of $C_N$ under the mapping $x\mapsto x+\frac{\rho}{x}$ eventually almost surely. 
\end{proof}

\subsection{Product of elliptic matrices}\label{subsection4.2}

Finally we give the proof of Theorem \ref{outliersofproducts}. The proof relies on a standard linearization argument.

\begin{proof}[\proofname\ of Theorem \ref{outliersofproducts}]

We consider the following two linearization matrices
$$
\mathcal{X}_N:=\begin{pmatrix}
    0&(d_{N,1})^{-1/2}X_N^1&\quad&\quad& 0\\
    0&0&(d_{N,2})^{-1/2}X_N^2 &\quad&0\\
    \quad&\quad&\ddots&\ddots&\quad\\
    0&\quad&\quad&0&(d_{N,m-1})^{-1/2}X_N^{m-1}\\(d_{N,m})^{-1/2}X_N^m&\quad&\quad&\quad& 0
\end{pmatrix}.
$$
and the $mN\times mN$ matrix
$$
\mathcal{A}_N:=\begin{pmatrix}
    0&A_N^1&&& 0\\
    0&0&A_N^2 &&0\\
    &&\ddots&\ddots&\\
    0&&&0&A_N^{m-1}\\A_N^m&&&& 0
\end{pmatrix}.
$$

A standard block-determinant computation (see \cite{coston2020outliers}, Proposition 4.1) gives
\[
\begin{aligned}
\det\bigl(z\mathbf{1}_{mN}-\mathcal{X}_N\bigr)
 &=\det\bigl(z^m\mathbf{1}_N-D_N^m\bigr),\\
\det\bigl(z\mathbf{1}_{mN}-\mathcal{A}_N\bigr)
 &=\det\bigl(z^m\mathbf{1}_N-A_N\bigr),\\
\det\bigl(z\mathbf{1}_{mN}-(\mathcal{X}_N+\mathcal{A}_N)\bigr)
 &=\det\bigl(z^m\mathbf{1}_N-D_N^{m,1}\bigr).
\end{aligned}
\]
Thus the nonzero spectrum of each linearization consists of the $m$-th roots of the corresponding product spectrum, with algebraic multiplicities.

Let $\zeta_m:=\{\omega\in\mathbb{C}:\omega^m=1\}$. In particular, the $mj$ eigenvalues of $\mathcal{A}_N$ corresponding to the $j$ outlying eigenvalues of $A_N$ form the multiset
\[
\left\{(\lambda_i(A_N))^{1/m}\omega:1\leq i\leq j,\ \omega\in\zeta_m\right\},
\]
where $(\lambda_i(A_N))^{1/m}$ denotes any one $m$-th root of $\lambda_i(A_N)$.

For the matrix $\mathcal{X}_N$, we readily check that whenever $m\geq 2$, $$\mathbb{E}[\mathcal{X}_N(\mathcal{X}_N)^*]=\mathbb{E}[(\mathcal{X}_N)^*\mathcal{X}_N]=\mathbf{1},\quad \mathbb{E}[(\mathcal{X}_N)^2]=0.$$ 

Then we can apply Theorems \ref{gaussianmastertheorem} and \ref{nongaussianmastertheorem} to $\mathcal{X}_N$ perturbed by the deterministic matrix $\mathcal{A}_N$ after checking via \eqref{4.6} and \eqref{4.7} that the quantitative bound conditions hold (when all entries of $\mathcal{X}_N$ are Gaussian in case (1) or bounded in case (2)), and conclude that for any $\epsilon>0$ a.s. for $N$ large there are no eigenvalues of $\mathcal{X}_N$ in $\mathbb{C}\setminus (1+\epsilon)^\frac{1}{m}\mathbb{D}$, so that  there are no eigenvalues of $D_N^m$ outside $(1+\epsilon)\mathbb{D}$. For the finite-rank perturbation, we use the immediate variant of part (2) of these master theorems for $\rho=0$: its proof is unchanged when the three neighborhood widths are arbitrary fixed numbers $0<\eta_1<\eta_2<\eta_3$, since it uses only their strict separation. Set
\[
\eta_q:=(1+q\epsilon)^{1/m}-1,\qquad q=1,2,3.
\]
The annular hypothesis on $A_N$ then implies the corresponding hypothesis for the nonzero eigenvalues of $\mathcal A_N$, which are the $m$-th roots of the nonzero eigenvalues of $A_N$.

Moreover, by this variant, almost surely for $N$ sufficiently large, there are precisely $mj$ outlying eigenvalues of $\mathcal{X}_N+\mathcal{A}_N$ in $\mathbb{C}\setminus (1+2\epsilon)^\frac{1}{m}\mathbb{D}$, and after relabeling they converge to $(\lambda_i(A_N))^\frac{1}{m}\zeta_m(1+o(1))$ for each $i=1,\cdots,j$. This implies that almost surely for $N$ large, there are precisely $j$ outlying eigenvalues of $D_N^{m,1}$ in $\mathbb{C}\setminus (1+2\epsilon)\mathbb{D}$, and after relabeling, these eigenvalues converge to $(\lambda_i(A_N))(1+o(1))$ for each $i=1,\cdots,j$. 

Finally, when entries of $\mathcal{X}_N$ have a bounded $p$-th moment, we may use a similar truncation argument as in the proof of Theorem \ref{theorem1.1} and \ref{TheEllipticTheorem}. The adaptations are straightforward and hence omitted.\end{proof}

\subsection{The case of self-loops}Finally, we present the proof of Theorem \ref{theoremlines404}.

\begin{proof}[\proofname\ of Theorem \ref{theoremlines404}]
    Let $H:=(d_N)^{-1/2}X_N$, then we have
    $$
\mathbb{E}[HH^*]=\mathbb{E}[H^*H]=\mathbf{1},\quad \mathbb{E}[H^2]=\frac{d_N-1}{d_N}\rho\mathbf{1}+\frac{1}{d_N}\mathbb{E}[\xi^2]\mathbf{1}.
    $$By definition of $d_N$ we can find $D_N=\frac{\mathbb{E}[\xi^2]-\rho}{d_N}\in\mathbb{C}$ such that $\lim_{N\to\infty}(\log N)D_N=0$ and $\mathbb{E}[H^2]=(\rho+D_N)\mathbf{1}$. The bounds on \(\sigma_*(H),v(H),\widetilde {v}(H)\), and in the bounded case \(R(H)\),
follow from Lemma \ref{lem:block-parameter} applied to the two-dimensional off-diagonal edge blocks
together with the one-dimensional diagonal blocks \(d_N^{-1/2}\xi E_{xx}\). Then we directly apply Theorem \ref{gaussianmastertheorem} or Theorem \ref{nongaussianmastertheorem}. 
\end{proof}

\appendix

\section{Proof of concentration inequalities}
\label{sectionappendixa}
In this appendix we outline a proof of Theorem \ref{resolventconcentrationgauss} and \ref{generalconcentration}; the proofs are essentially an adaptation of the version in \cite{brailovskaya2022universality}. 

\begin{proof}[\proofname\ of Theorem \ref{resolventconcentrationgauss}]

We write the Gaussian model $G$ as
$$
G:=A_0+\sum_{i=1}^n g_iA_i
$$ where $g_1,\cdots,g_n$ are independent mean $0$, variance $1$ Gaussian variables and $A_i\in M_N(\mathbb{C})_{sa}$ are fixed self-adjoint matrices. Define $f:\mathbb{R}^n\to\mathbb{C}$ as
$$
f(x):=\langle u,(z\mathbf{1}-A_0-\sum_{i=1}^nx_iA_i)^{-1}v\rangle.
$$
Set $L:=\frac{\sigma_*(G)}{(\operatorname{Im} z)^2}$ and, for $\theta\in\{1,i\}$, let $F_\theta(x):=\operatorname{Re}(\theta f(x))$. As in the proof of \cite{brailovskaya2022universality}, Lemma 5.5, each $F_\theta$ is $L$-Lipschitz. Hence the real-valued Gaussian concentration inequality from \cite{10.1093/acprof:oso/9780199535255.001.0001}, Theorem 5.6, gives
\[
\mathbb{P}\left(\left|F_\theta(g)-\mathbb{E}F_\theta(g)\right|\geq Lx\right)\leq 2e^{-x^2/2},\qquad \theta\in\{1,i\},
\]
where $g=(g_1,\ldots,g_n)$. Since $|w|\geq\sqrt{2}Lx$ implies $|\operatorname{Re}w|\geq Lx$ or $|\operatorname{Im}w|\geq Lx$, a union bound proves the stated estimate.\end{proof}

\begin{proof}[\proofname\ of Theorem \ref{generalconcentration}]
Fix unit vectors $u,v\in\mathbb{C}^N$. For a self-adjoint matrix $M$, set
\[
Y(M):=\langle u,(z\mathbf{1}-M)^{-1}v\rangle,\qquad
F_\theta(M):=\operatorname{Re}\bigl(\theta Y(M)\bigr),\quad \theta\in\{1,i\}.
\]
We apply the following real-valued argument to each $F_\theta$. We will adapt the proof of \cite{brailovskaya2022universality} and only give a sketch. Some estimates we present are suboptimal, but we will use them as we can simply quote the computations from the cited work.

Let $W$ be of the form \eqref{generalmodels} and consider $(Z_1',\cdots,Z_n')$ to be some independent copy of $(Z_1,\cdots,Z_n)$. Let $W^{\sim i}:=Z_0+\sum_{j\neq i}Z_j+Z_i'$. Since
$|F_\theta(W)-F_\theta(W^{\sim i})|\leq |Y(W)-Y(W^{\sim i})|$, the same resolvent-identity computation gives, for each $\theta\in\{1,i\}$,
$$
\sum_{i=1}^n \bigl(F_\theta(W)-F_\theta(W^{\sim i})\bigr)_+^2
\leq \sum_{i=1}^n|Y(W)-Y(W^{\sim i})|^2\leq T,
$$where we define $T$ via
$$
T:=\frac{2}{(\operatorname{Im}z)^4}\sup_{\|v\|=\|w\|=1}\sum_{i=1}^n|\langle v,(Z_i-Z_i')w\rangle|^2+\frac{8}{(\operatorname{Im}z)^6}R(W)^2\|\sum_{i=1}^n(Z_i-Z_i')^2\|.
$$

Then we can estimate as in Lemma 5.8 of \cite{brailovskaya2022universality},
$$
\mathbb{E}[T]\lesssim \frac{\sigma_*(W)^2}{(\operatorname{Im}z)^4} +\frac{R(W)\mathbb{E}\|W-\mathbb{E}W\|}{(\operatorname{Im}z)^4}+\frac{R(W)^2\mathbb{E}\|W-\mathbb{E}W\|^2}{(\operatorname{Im}z)^6}.
$$ $T$ has the following self-bounding property: consider
$$
T^{\sim i}:=\frac{2}{(\operatorname{Im}z)^4}\sup_{\|v\|=\|w\|=1}\sum_{j\neq i}|\langle v,(Z_j-Z_j')w\rangle|^2+\frac{8}{(\operatorname{Im}z)^6}R(W)^2\|\sum_{j\neq i}^n(Z_j-Z_j')^2\|,
$$
    then by \cite{brailovskaya2022universality}, Lemma 5.9, $T^{\sim i}\leq T$ and 
    $$
\sum_{i=1}^n(T-T^{\sim i})^2\leq\left\{\frac{16R(W)^2}{(\operatorname{Im}z)^4}+\frac{64R(W)^4}{(\operatorname{Im}z)^6}\right\}T.
    $$

    If $R(W)=0$, then $Z_i=0$ almost surely for every $i\geq 1$, so $W=Z_0$ is deterministic and the asserted bound is immediate. We may therefore assume that $R(W)>0$.

    From the self-bounding property of $T$ we get $$ \log\mathbb{E}[e^{T/a}]\leq\frac{2}{a}\mathbb{E}[T],\quad a=\frac{16R(W)^2}{(\operatorname{Im}z)^4}+\frac{64R(W)^4}{(\operatorname{Im}z)^6}.
    $$
    Then for $0\leq\lambda\leq a^{-1/2}$, by exponential Poincaré inequality and Chernoff bound, we have, for each $\theta\in\{1,i\}$,
$$ \mathbb{P}\left[F_\theta(W)-\mathbb{E}F_\theta(W)\geq\sqrt{8\mathbb{E}Tx}+\sqrt{a}x
\right]\leq e^{-x}.
    $$

We can also prove via a refinement of \cite{10.1093/acprof:oso/9780199535255.001.0001} Theorem 6.16 that for any $0\leq\lambda\leq (2a)^{-\frac{1}{2}},$
$$
\log\mathbb{E}\left[e^{-\lambda\{F_\theta(W)-\mathbb{E}F_\theta(W)\}}\right]\leq\frac{4\mathbb{E}[T]\lambda^2}{1-\lambda\sqrt{2a}}.
$$Then Chernoff's inequality implies the lower tail bound for any $x>0$
$$ \mathbb{P}\left[F_\theta(W)-\mathbb{E}F_\theta(W)\leq-4\sqrt{\mathbb{E}Tx}-\sqrt{2a}x
\right]\leq e^{-x}.
    $$
Set $L(x):=4\sqrt{\mathbb{E}Tx}+\sqrt{2a}x$. The two tails for $\theta=1,i$ and a union bound give
\[
\mathbb{P}\left[\left|Y(W)-\mathbb{E}Y(W)\right|\geq\sqrt{2}L(x)\right]\leq4e^{-x}.
\]
Indeed, the two centered projections are the real and negative imaginary parts of $Y(W)-\mathbb{E}Y(W)$. Using the displayed bounds on $\mathbb{E}T$ and $a$, then rescaling $x$ by a universal constant, absorbs the numerical factors and $4e^{-x}$ into the universal constant $C$ in the statement. This proves the claimed estimate.
    \end{proof}

\printbibliography

\end{document}